\documentclass{article}

\usepackage{amsthm,amsmath,amssymb}
\RequirePackage
{natbib}
\RequirePackage[colorlinks,citecolor=blue,urlcolor=blue]{hyperref}
\usepackage{graphicx}
\usepackage{ulem}
\usepackage{dsfont}

\newcommand{\E}{\mathbb{E}}
\renewcommand{\P}{\mathbb{P}}
\renewcommand{\S}{\mathbb{S}}
\newcommand{\R}{\mathbb{R}}
\newcommand{\Z}{\mathbb{Z}}

\newcommand{\C}{\mathbb{C}}
\newcommand{\Var}{\mbox{Var}}
\DeclareMathOperator{\argmin}{argmin\,}

\newtheorem{theorem}{Theorem}
\newtheorem{proposition}[theorem]{Proposition}
\newtheorem{corollary}[theorem]{Corollary}
\newtheorem{lemma}[theorem]{Lemma}
\newtheorem{remark}{Remark}
\newtheorem{assumption}{Assumption}

\author{
{Claire Lacour\footnote{claire.lacour@u-pem.fr, 
{LAMA, Univ Gustave Eiffel, Univ Paris Est Creteil, CNRS, F-77447 Marne-la-Vall\'ee, France}} }
and 
{Thanh Mai Pham Ngoc\footnote{thanh.pham\_ngoc@math.u-psud.fr,
{Universit\'e Paris-Saclay, CNRS, Laboratoire de math\'ematiques d'Orsay, 91405 Orsay, France}}
}}

\date{}

\begin{document}


\title{Semiparametric inference for mixtures of circular data
}
\maketitle


\begin{abstract}
We consider $X_1, \dots, X_n$ a sample of data on the circle $\mathbb{S}^1$, whose distribution is a two-component mixture. Denoting $R$ and $Q$ two rotations on $\mathbb{S}^1$, the density of the $X_i$'s is assumed to be
$g(x)=p f(R^{-1} x)+(1-p) f(Q^{-1} x)$, where $p\in (0,1)$ and $f$ is an unknown density on the circle. In this paper we estimate both the parametric part $\theta=(p,R,Q)$ and the nonparametric part $f$. The specific problems of identifiability on the circle are  studied. A consistent estimator of $\theta$ is introduced and its asymptotic normality is proved.  We propose a Fourier-based estimator of $f$ with a penalized criterion to choose the resolution level. 
We show that our adaptive estimator is optimal from the oracle and minimax points of view when the density belongs to a Sobolev ball.
Our method is illustrated by numerical simulations.
\end{abstract}


\section{Introduction}

Circular data are collected when the topic of interest is a  direction or  a time of day. These particular data appear in many 
applications: earth sciences (e.g. wind directions), medicine (e.g. circadian rhythm), ecology (e.g. animal movements), forensics (crime incidence).
Different surveys on statistical methods for circular data can be found: \cite{Mardia},
 \cite{JammaSengup}, \cite{LeyVerdebout} or more recently \cite{PewseyGarcia}. In the present work, we consider a mixture model with two components equal up to a rotation.  We observe $X_1, \dots, X_n$ a sample of data on $\S^1$ with probability distribution function:
\begin{equation}\label{model}
g(x)=p_0 f(R_0^{-1} x)+(1-p_0) f(Q_0^{-1} x)=p_0 f( x-\alpha_0)+(1-p_0) f(x-\beta_0).
\end{equation}
In the right hand side we have identified $f:\S^1\to \R$ and its periodized version on $\R$. 
Here $R_0$ and $Q_0$ are two unknown rotations of the circle. 
$R_0$ is a rotation with angle $\alpha_0$ and $Q_0$ is a rotation with angle $\beta_0$.  The aim is to estimate both $\theta_0=(p_0,\alpha_0,\beta_0)$ and the nonparametric part $f$.  

Bimodal circular data are commonly encountered in many scientific fields, for instance in climatology, animal orientations or in earth sciences.
For the analysis of wind directions, see  \cite{HernandezScarpa} and for  animal orientations, the dragonflies data set presented in \cite{Batschelet}. %
In geosciences, one can cite the 
  cross-bed orientations data set obtained in the middle Mississipian Salem Limestone of central Indiana and which was presented by the Seminar Sedimentation (\cite{Sedimentation}). Last but not least, the paper of \cite{Lark} analyzes some geological data sets  and clearly favours for some of them a two component mixture of von Mises distributions.

\medskip

Mixture models for describing multimodal circular data date back to \cite{Pearson} and have been largely used since then.
An important case in the literature is the mixture of two von Mises distributions which has been explored in numerous works. Let us cite among others papers by \cite{Bartels}, \cite{Spurr} or \cite{ChenLiFu}. From a practical point of view, algorithms have also been proposed to deal with mixture of two von Mises distributions, including maximum likelihood algorithms by \cite{JonesJames} or a characteristic function based procedure by \cite{SpurrKout}. Note that on the unit hypersphere, \cite{Banerjee} investigated clustering methods for  mixtures of von Mises Fisher distributions. In our framework, we shall not assume any parametric form of the density and hence the model is said to be semiparametric. {To the best of our knowledge, this is the first work devoted to the study of the semiparametric mixture model for circular data}.  {This semiparametric model is more complex and intricate than the usual parametric one encountered in the circular literature}. In the spherical case, \cite{KimKoo} studied the general mixture framework for a location parameter but assuming that the nonparametric part $f$ is known. On the real line, this semiparametric model has been studied by \cite{bordes2006}, \cite{Hunterwang07}, \cite{butucea2014} or \cite{GassiatRousseau} for dependent latent variables. For the multivariate case, see for instance  \cite{HallZhou}, \cite{HallNeeman}, \cite{GassiatRousseauVernet}, \cite{Hohmann}. When dealing with the specific case of one of the two components being parametric, one refers to work by \cite{MaYao} and references therein.  

Note that we can rewrite model (\ref{model}) as 
\begin{equation}\label{modelprime}
X_i = Y_i+\varepsilon_i \: \pmod{2\pi},\qquad i=1,\dots,n,
\end{equation}
where $Y_i$ has density $f$ and $\varepsilon_i$ is a Bernoulli angle, which is equal to $\alpha_0$ with probability $p_0$ and $\beta_0$ otherwise. Accordingly, model (\ref{model}) can be viewed as a circular convolution model with unknown noise operator $\varepsilon$. The circular convolution model has been studied by  \cite{Golden} in the case of known noise operator whereas \cite{JohannesSchwarz} dealt with unknown error distribution but have at their disposal an independent sample of the noise to estimate this latter. It is worth pointing out that 
\cite{Golden} and \cite{JohannesSchwarz} made the usual assumptions on the decay of the Fourier coefficients of the density of $\varepsilon$, whereas in model \eqref{model} the Fourier coefficients are not decreasing.

\medskip

Identifiability questions are at the heart of the theory of mixture models and the circular context is no exception. Thus, our first task is to study the identifiability of the model. From a mathematical point of view, the topology of the circle makes the problem very different from the linear case. In the circular parametric case, \cite{Fraser} obtained  identifiability results for the von Mises distributions, which were extended in \cite{Kent83} to generalized von Mises distributions while \cite{holzmann2004}) focused on wrapped distributions, basing their analysis on the Fourier coefficients. Here, the Fourier coefficients turn out to be very useful as well but the nonparametric paradigm makes the study quite different and intricate.  Our identifiability results are obtained under mild assumptions on the Fourier coefficients. We require that the coefficients are real which can be related to the usual symmetry assumption in mixture models (see for instance \cite{Hunterwang07}) and we impose that only the first 4 coefficients do not vanish. Interestingly enough, some not intuitive phenomena appear. {A striking case} occurs when the angles $\alpha_0$ and $\beta_0$ are distant from $2\pi/3$,  model \eqref{model} is then nonidentifiable which is quite surprising at first sight.

 Once the identifiability of the model is obtained, we resort to a contrast function in the line of \cite{butucea2014} to estimate the Euclidian parameter $\theta_0$. In that regard, we prove the consistency of our estimator and an asymptotic normality result.  Thereafter, for the estimation of the nonparametric part, a penalized empirical risk estimation method is used. The estimator of the density turns out to be adaptive (meaning that it does not require the specification of the unknown smoothness parameter), {a property which was not reached so far for this semiparametric model even in the linear case}. The procedure devised is hence relevant for practical purposes. We prove an oracle inequality and minimax rates are achieved by our estimator for Sobolev regularity classes. Eventually, a numerical section shows the good performances of the whole estimation procedure.

The paper is organized as follows. Section 2 is devoted to the identifiability of the model. Section 3 tackles the estimation of the parameter $\theta_0$ whereas Section 4 focuses on the estimation of the nonparametric part. Finally Section 5 presents numerical implementations of our procedure. Proofs are gathered in Section 6.

\section{Identifiability}

In this section, to keep the notation as light and clear as possible, we drop the subscript $0$ in the parameters. For any function $g$ and any angle $\alpha$, denote $g_\alpha(x):=g(x-\alpha)$. For any complex number $a$, $\overline a$ is the complex conjugate of $a$. For any integrable function $\phi:  \S^1\to \R$, we denote for any $l\in \Z$, 
$\phi^{\star l}=\int_{\S^1} \phi(x) e^{-i l x} \frac{dx}{2\pi}$, the Fourier coefficients.  {Note also that we use notation $f$ and $f'$ for two densities, where $f'$ is not the derivative of $f$.}

\medskip

Let us now study the identifiability of our model (\ref{model})
where the data have density $ p f( x-\alpha)+(1-p) f( x-\beta)$. First, it is obvious that if $p=0$, $\alpha$ is not identifiable, and  if $p=1$, $\beta$ is not identifiable. In the same way, $p$ is not identifiable if $\alpha=\beta$. 
Moreover, as explained in \cite{Hunterwang07} for a translation mixture on the real line,  the case $p=1/2$ has to be avoided. Indeed, denoting $g$ a  
density and for instance
$f=\frac12 g_1+\frac12 g_{-1}$
and 
$f'=\frac12 g_2+\frac12 g_{-2}$
we have 
$$f_1+f_5=f'_{2}+f'_{4}.$$
In addition, it is well known that, in such a mixture model, $(p, \alpha, \beta)$ cannot be distinguished from 
$(1-p,\beta,\alpha)$: it is the so-called \textit{label switching} problem.
So we will assume that $p\in (0,1/2)$ (for mixtures on $\R$ it is assumed alternatively that $\alpha<\beta$ but ordering angles is less relevant). 

Now let us study the specific problems of identifiability on the circle, that do not appear on $\R$. 
First, if $f$ is the uniform probability, the model is not identifiable, so we have to exclude this case. Another case to exclude is the case of $\delta$-periodic functions. 
 Indeed in this case 
$f_{\alpha}=f_{\alpha+\delta}$.
These functions have the property that $f^{\star l}=0$ for all $l\notin (2\pi/\delta)\Z$. So we will require that the Fourier coefficients of $f$ do not cancel out too much. Here we will  assume 
$$\text{ for all } l\in \{1,2,3,4\}, \qquad f^{\star l}\neq 0,\text{ and} \quad f^{\star l}=\overline{f^{\star l}} .$$
This last assumption  can be related to the symmetry of $f$. Indeed if $f$ is zero-symmetric then all its Fourier coefficients are real. Symmetry is a usual assumption in this mixture context, to distinguish between the translations of $f$:
for any $\delta\in \R$, $$p f( x-\alpha)+(1-p) f(x-\beta)=p f_\delta( x-\alpha+\delta)+(1-p) f_\delta(x-\beta+\delta)$$
More precisely, \cite{Hunterwang07} show that symmetry is a sufficient and necessary condition for identifiability of the model mixture on $\R$.
In the circle framework, it is natural to work with Fourier coefficients rather than Fourier transform as on $\R$. A lot of circular densities have their Fourier coefficients real, provided that their location parameter is $\mu=0$: for example the Jones-Pewsey density, which includes the cardioid, the wrapped Cauchy density, and the von Mises density.
Here we require the assumption only for the first 4 Fourier coefficients of $f$ (due to our proof), which is milder than symmetry. 

Let us now state our identifiability result under these assumptions. Note that \cite{holzmann2004} have studied the identifiability of this model when $f$ belongs to a parametric scale-family of densities, but here we face a nonparametric problem concerning $f$. 

\begin{theorem} \label{identifiability}
Assume that $\theta=(p, \alpha, \beta)$ and $\theta'=(p', \alpha', \beta')$ belong to 
$$\left\{(p, \alpha, \beta)\in (0,1/2)\times \S^1\times \S^1, \quad \alpha\neq \beta\pmod{2\pi}\right\}$$
and that 
$f,f'$ belongs
 to  
$$ \left\{f: \S^1\to \R  \text{ density such that, for all } l\in \{1,2,3,4\}, \; f^{\star l}\in \R\backslash \{0\}\right\}.$$  
Suppose 
$pf_{\alpha}+(1-p)f_{\beta}=p'f'_{\alpha'}+(1-p')f'_{\beta'}$. Then
\begin{enumerate}
\item either $(p',\alpha', \beta')$=$(p,\alpha, \beta)$ and $f'=f$,

\item or $(p', \alpha', \beta')=(p, \alpha+\pi, \beta+\pi)$ and 
$f'=f_{\pi}$ 

\item or if $\beta-\alpha=\pi \pmod{2\pi}$, then   $f'$ is a linear combination of $f$ and $f_{\pi}$, 
and 
either 
 $(\alpha',\beta')=(\alpha, \beta)$,  
or $(\alpha',\beta')=(\beta, \alpha)$.

\item or  if $\beta-\alpha=\pm 2\pi/3 \pmod{2\pi}$, then  $f'$ is a linear combination of 
$f_{\pi/3}, f_{-\pi/3}, f_{\pi}$
and $p'=(1-2p)/(2-3p)$ and 

(a) if $\beta-\alpha=2\pi/3 $,
  $(\alpha',\beta')=(\alpha+\pi,\beta-\pi/3)$
  {or $(\alpha',\beta')=(\alpha,\beta+2\pi/3)$}, 

(b) if $\beta-\alpha=-2\pi/3 $,
  $(\alpha',\beta')=(\alpha+\pi,\beta+\pi/3)$ 
 {or $(\alpha',\beta')=(\alpha,\beta-2\pi/3)$.}
\end{enumerate} 
\end{theorem}

Case 2.  arises from a  specific feature of circular distributions: if $f$ is symmetric with respect to  0 then it is  symmetric with respect to $\pi$. Unlike the real case, a symmetry assumption does not exclude the case $f'(x)=f(x-\pi)$. To bypass this we could assume for instance
$ f^{\star 1}> 0.$
Indeed %
for each $l\in \Z$, $(f_{\pi})^{\star l}=f^{\star  l}(-1)^l$, so the Fourier coefficients of $f$ and $ f_\pi$ have opposite sign for any odd $l$.
 With our assumption, we recover among $f$ and $ f_\pi$ the one with positive first Fourier coefficient, i.e. with positive  mean resultant length.
Nevertheless our estimation procedure begins with the parametric part so that this assumption concerning only the nonparametric part will not allow us to distinguish $\alpha$ from $\alpha+\pi$ in this first parametric estimation step. That is why we rather choose to assume that $\alpha$ and  $\beta$ belong to $[0,\pi) \pmod{\pi}$.

Case 3. concerns bipolar data since $\alpha$ and $\beta$ are diametrically opposed (separated by $\pi$ radians). In this case $\alpha$ and $\beta$ are identifiable, but $p$ and $f$ not. 
{
Indeed, for any density $f$ and any $0<p'\leq p <1/2$, we can find $q\in (0,1]$ such that  $f'=qf+(1-q)f_{\pi}$ verifies $pf_{\alpha}+(1-p)f_{\beta}=p'f'_{\alpha'}+(1-p')f'_{\beta'}$.
}
{  
Thus our result demonstrates that bimodal data sets with opposite modes lead to non-identifiability issues, and this highlights a fundamental issue in considering a too large class of possible densities.
}

{Let us now discuss the case 4., which is the most curious
({we shall only comment the first case (a)}, the other is similar). 
{Let us set 
$$f'(x)=(1-p)f\left(x-\frac{\pi}3\right)+ (1-p) f\left(x+\frac{\pi}3\right)+ (2p-1) f(x-\pi).$$
This function is symmetric if $f$ is symmetric, verifies $\int_{\S^1} f'=1$ and may be positive for some values of $p$ (depending on $f$): see  Figure~\ref{figf'}.}
\begin{figure}[h]
\begin{center}
\includegraphics[scale=0.6]{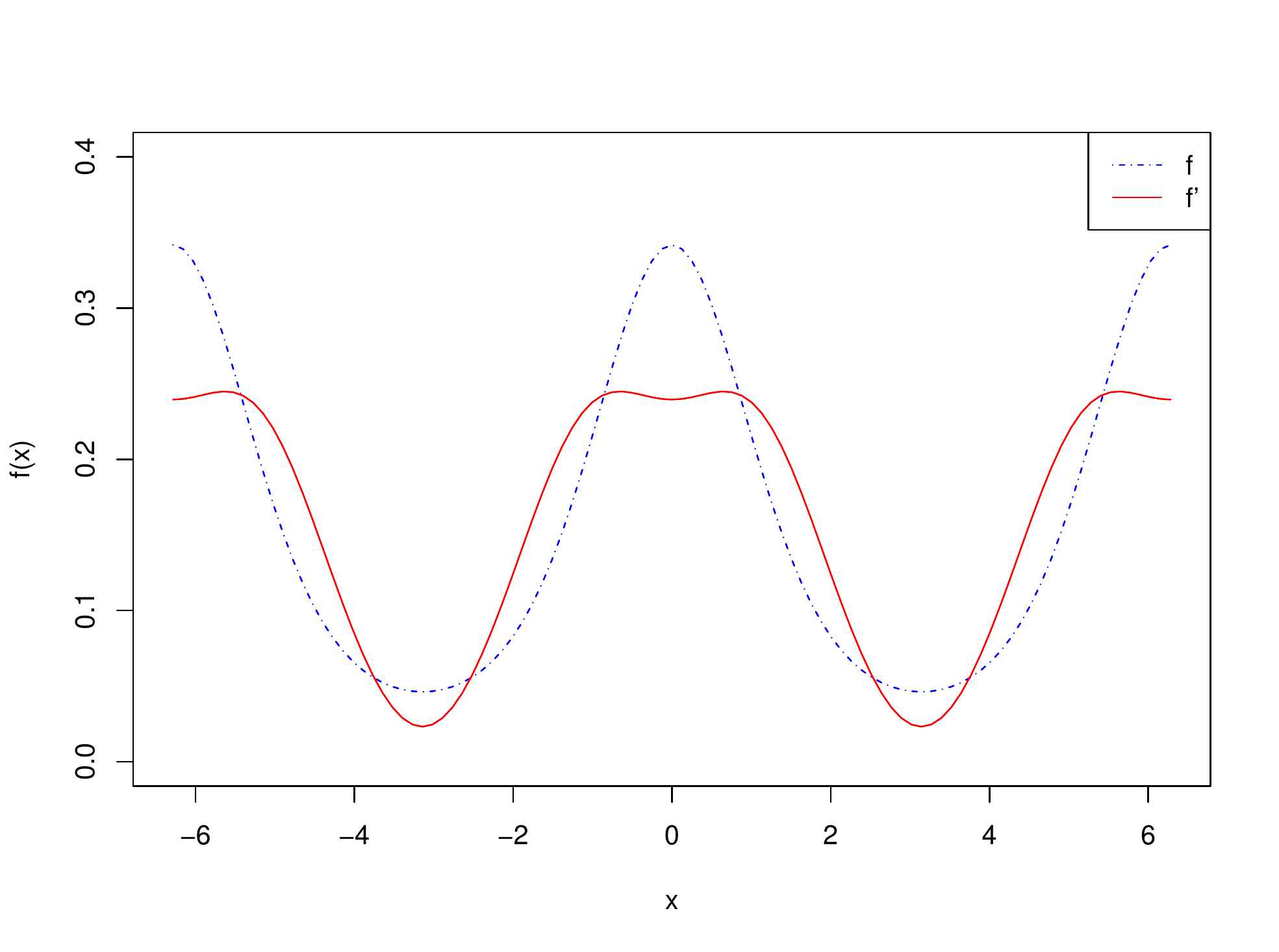}
\end{center}
\caption{Plot of a circular density $f$  (dashed blue), and of 
$f'=(1-p)f_{\frac{\pi}3}+ (1-p) f_{-\frac{\pi}3}+ (2p-1) f_\pi$ 
(solid red).
Here $f$ is  the von Mises density  with mean 0 and concentration 1. 
{In this case, $f'$ is positive as soon as $p\geq 0.36$, here $p=0.4$.}
}
\label{figf'}
\end{figure}
Then we can write $f'_{\pi/3}$:
$$f'\left(x-\frac\pi{3}\right)=(1-p)f\left(x-\frac{2\pi}3\right)+ (1-p) f\left(x\right)+ (2p-1) f\left(x-\frac{4\pi}3\right),$$
as well as $f'_{\pi}$:
$$f'(x-\pi)=(1-p)f\left(x-\frac{4\pi}3\right)+ (1-p) f\left(x-\frac{2\pi}3\right)+ (2p-1) f(x).$$
Hence a mixture of  $f'_{\pi}$ and $f'_{\pi/3}$  gives a mixture of $f(x), f(x-\frac{2\pi}3), f(x-\frac{4\pi}3)$: 
\begin{eqnarray*}p'f'(x-\pi)+(1-p')f'\left(x-\frac\pi{3}\right)&=&
[p'(2p-1)+(1-p')(1-p)]f(x)\\
&&+[p'(1-p)+(1-p')(1-p)] f\left(x-\frac{2\pi}3\right)\\
&&+[p'(1-p)+(1-p')(2p-1)]f\left(x-\frac{4\pi}3\right)
\end{eqnarray*}
 If now $p'=(1-2p)/(2-3p)$, then $p'(1-p)+(1-p')(2p-1)=0$ and the third component  $f(x-\frac{4\pi}3)$ vanishes. Thus
$$p' f'(x-\pi)+(1-p')f'\left(x-\frac{\pi}3\right) =p f( x )+(1-p)f\left(x-\frac{2\pi}3\right).$$}

In such a particular case, we cannot identify $\theta$ nor $f$. However this happens only when $\beta-\alpha=\pm 2\pi/3$. So,
to exclude these cases, we will now assume $\beta \neq \alpha \pmod{2\pi/3}$.

\bigskip 

Finally, we shall assume 
that $f\in \mathcal{F}$ with some assumptions for $\mathcal{F}$:
\begin{assumption}\label{hypf} 
$$\mathcal F\subset \left\{f: \S^1\to \R  \text{ density s.t. for all } l\in \{1,2,3,4\}, \quad f^{\star l}\in \R\backslash \{0\}\right\}$$
\end{assumption}
or 
\begin{assumption}\label{hypf2} 
$$\mathcal F\subset \left\{f: \S^1\to \R  \text{ density s.t. for all } l\in \{1,2,3,4\}, \quad f^{\star l}\in \R\backslash \{0\},\quad f^{\star 1}>0\right\}$$
\end{assumption}
and we shall assume that $\theta\in \Theta$ with some assumptions for $\Theta$:
\begin{assumption}\label{hyptheta2}
$$\Theta \subset \left\{(p, \alpha, \beta)\in \left(0, \frac12\right)\times \S^1\times  \S^1, \quad \alpha\neq \beta \pmod{\pi, 2\pi/3}\right\}$$
\end{assumption}
where $\alpha\neq \beta  \pmod{2\pi/3,\pi}$ means 
$\beta-\alpha\notin\{-\frac{2\pi}{3},0,\frac{2\pi}{3},\pi\}+ 2\pi\Z$, or
\begin{assumption}\label{hyptheta}
$$\Theta \subset \left\{(p, \alpha, \beta)\in \left(0, \frac12\right)\times [0,\pi)\times  [0,\pi), \quad \alpha\neq \beta \pmod{2\pi/3}\right\}$$
\end{assumption}
Note that Assumption \ref{hyptheta} implies Assumption \ref{hyptheta2}, and Assumption \ref{hypf2} implies Assumption \ref{hypf}. We can write the following result.
\begin{corollary} 
Under Assumptions \ref{hypf} and \ref{hyptheta}, or under Assumptions \ref{hypf2} and \ref{hyptheta2}, model \eqref{model} is identifiable.
Under Assumptions \ref{hypf} and \ref{hyptheta2}, model \eqref{model} is identifiable modulo $\pi$, that is to say that if 
$pf_{\alpha}+(1-p)f_{\beta}=p'f'_{\alpha'}+(1-p')f'_{\beta'}$
then $p'=p$ and 
 either $(\alpha', \beta')$=$(\alpha, \beta)$ and $f'=f$,
or $(\alpha', \beta')=(\alpha+\pi, \beta+\pi)$ and 
$f'=f_{\pi}$.
\end{corollary}

Moreover, the proof of Theorem~\ref{identifiability} provides the following statement.

\begin{lemma} \label{EqMtheta}
Under Assumption~\ref{hyptheta2}, denoting 
$
M^l(\theta):= pe^{-i \alpha l }+ (1-p)e^{-i\beta l}
$, for all $\theta,\theta' \in \Theta$, 
$$\forall 1 \leq l \leq 4, \; \Im\left(M^l(\theta')\overline{M^l(\theta)}\right) =0 \Leftrightarrow \theta'=\theta\text{ or } \theta'{=}\theta+\pi.$$
where $\theta'{=}\theta+\pi$ means $(p', \alpha', \beta')=(p, \alpha+\pi, \beta+\pi)$.
\end{lemma}

\section{Estimation for the parametric part}

Now, let us denote for all $l \in \Z$
\[
M^l(\theta):= pe^{-i \alpha l }+ (1-p)e^{-i\beta l}.
\]
In model (\ref{model}) the Fourier coefficients of $g$ satisfy for any $l$: 
 \begin{equation*}\label{coeff-g}
g^{\star l} = (p_0e^{-i \alpha_0 l} + (1-p_0) e^{-i  \beta_0 l} )f^{\star l }.
\end{equation*}

Thus  $g^{\star l} = M^l(\theta_0)f^{\star l } $ and the previous lemma gives that $\theta=\theta_0$ (or $\theta_0+\pi$) if and only if, for each $l\in \{1,\dots, 4\}$,  
$$ \Im\left(M^l(\theta_0)\overline{M^l(\theta)}\right) =0 \Leftrightarrow \Im\left(g^{\star l} \overline{M^l(\theta)}\right) =0 $$
using that $f^{\star l }$ are non-zero real numbers. 
This invites us to consider 
\[
S(\theta):=  
\sum_{l =-4}^{4} \left ( \Im \left ( g^{\star l}\overline{{M}^l(\theta)} \right ) \right )^2 = 
\sum_{l =-4}^{4} \left ( \Im \left (   g^{\star l}\{pe^{i \alpha l} + (1-p) e^{i \beta l} \} \right )   \right )^2.
\] 
Note that $g^{\star0}\overline{{M}^{0}(\theta)}=1/(2\pi)$ and that $ \Im \left ( g^{\star (-l)}\overline{{M}^{-l}(\theta)} \right ) =
\Im \left (\overline{ g^{\star l}}{{M}^l(\theta)} \right ) =-\Im \left ( g^{\star l}\overline{{M}^l(\theta)} \right ) $ so that 
we can also write \[
S(\theta)=  
2\sum_{l =1}^{4} \left ( \Im \left ( g^{\star l}\overline{{M}^l(\theta)} \right ) \right )^2 .
\]

The empirical counterpart of $S(\theta)$ is 
\begin{eqnarray*}
\tilde S_n(\theta) &=&  \sum_{l=-4}^4 \left ( \Im \left (  \widehat {g^{\star l}}\overline{ M^l(\theta)} \right  ) \right )^2\\
&=& \sum_{l=-4}^4 \left (\Im \left ( \frac {1}{2 \pi  n} \sum_{k=1}^n  e^{-i l X_k}\overline{ M^l(\theta)} \right )  \right)^2  \\
&=&\frac{1}{4 \pi^2n^2}  \sum_{l=-4}^4\sum_{1\leq k,j\leq n } \Im \left ( e^{i lX_k}{ M^l(\theta)}\right ) \Im \left (  e^{i lX_j}{ M^l(\theta)} \right )   .
\end{eqnarray*}
Next, we consider a slightly modified version of $ \tilde S_n(\theta)$ by removing the diagonal terms 

\begin{equation}\label{Sn}
S_n(\theta) = \frac{1}{ 4 \pi^2n(n-1) }  \sum_{l=-4}^{4} \sum_{1\leq k \neq j\leq n}    \Im \left (  e^{i lX_k}{ M^l(\theta)}\right ) \Im \left (  e^{i lX_j}{ M^l(\theta)} \right ) .
\end{equation}
Let us denote
\[
Z_{ k}^l(\theta):= \Im \left (  \frac{e^{i lX_k}}{2 \pi}  M^l(\theta) \right )\quad
\text{ and }\quad 
J^l(\theta) :=  \Im \left (\overline{g^{\star l}}{  M^l(\theta)}\right ).
\]
Hence
\[
S_n(\theta)= \frac{1}{n(n-1)} \sum_{l=-4}^{4}  \sum_{1\leq k \neq  j\leq n} Z_{k}^l(\theta) { Z_{j}^l(\theta) }.
\]
Note that we have 
$\E( Z_{k}^l(\theta) )= J^l(\theta),$
and $S_n(\theta)$ is an unbiased estimator of $S(\theta)$. 

Let the estimator of $\theta_0$ be

\begin{equation}\label{estimator_theta}
\hat \theta_n = \argmin_{\theta \in \Theta} S_n(\theta). 
\end{equation}
For this estimator we can prove the following consistency result.

\begin{theorem}\label{consistance}
Consider $\Theta$  a compact set included in 
$$\left\{(p, \alpha, \beta)\in \left(0, 1/2\right)\times \S^1\times \S^1, \quad \alpha\neq \beta  \pmod{2\pi/3,\pi}\right\}$$
 and the estimator $\hat \theta_n = \argmin_{\theta \in \Theta} S_n(\theta). $ We have 
 $\hat \theta_n \rightarrow  \theta_0 \pmod{\pi}$ in probability. 
\end{theorem}
The last convergence means that for all $\epsilon>0$, the probability $\P(\|\hat \theta_n-\theta_0\|\leq \epsilon \text{ or  } \|\hat \theta_n-\theta_0-\pi\|\leq \epsilon)$ tends to 1 when $n$ goes to $+\infty$, where $\|.\|$ denotes the Euclidean norm. 

\begin{proof} 
$\Theta$ is a compact set and $S$ is continuous. Lemma~\ref{lip} ensures that $S_n$ is Lipschitz hence uniformly continuous, and 
Proposition~\ref{variance} ensures that for all $\theta$,
$| S_n(\theta)-S(\theta)|$ tends to 0 in probability.
Then
it is sufficient to apply a classical Lemma to conclude. See the details in Section~\ref{proof:consistence} 
\end{proof}

From now on, we assume that $\Theta$ is a compact set included in 
$\left(0, \frac12\right)\times [0,\pi)\times [0,\pi)$, as in Assumption \ref{hyptheta}. 
Then, $\theta_0+\pi$ is excluded and under Assumption \ref{hyptheta}, $\hat \theta_n \rightarrow  \theta_0 $ in probability. 
Moreover this estimator is asymptotically normal. We denote $\dot  \phi(\theta)$ the gradient of any function $\phi$ with respect to $\theta=(p,\alpha,\beta)$, $\ddot  \phi(\theta)$ the Hessian matrix and  for any matrix $A$, we denote $A^\top$ its transpose.

\begin{theorem}\label{normasymp} 
Consider $\Theta$  a compact set included in $$\left\{(p, \alpha, \beta)\in \left(0, 1/2\right)\times [0,\pi)\times [0,\pi), \quad \alpha\neq \beta  \pmod{2\pi/3}\right\}$$
 and the estimator $\hat \theta_n = \argmin_{\theta \in \Theta} S_n(\theta) .$ Assume that $\theta_0\in \Theta$. 
 { Let $\mathcal{A}$ be the Hessian matrix of $S$ in $\theta_0$: $\mathcal{A}=\ddot S(\theta_0)= 2\sum_{l=-4}^4 \dot J^l(\theta_0) \dot J^l(\theta_0)^\top$. 
 Then, if $\mathcal{A}$ is invertible,
 $$
 \sqrt{n}(\hat \theta_n -\theta_0) \overset{d}{\longrightarrow} \mathcal{N}(0, \Sigma), 
 $$
 where $\Sigma= \mathcal{A}^{-1}V\mathcal{A}^{-1} $, 
 $V=4\E(UU^\top)$ and $U= \sum_{l=-4}^4 \dot J^l(\theta_0) Z^l_{1}(\theta_0) $. }
\end{theorem}

The proof can be found in Section~\ref{proof:normasymp}.
Note that $\mathcal{A}$ can  be estimated by $\ddot S(\hat \theta_n)$ and $V$ by 
$$\frac{4}{n^3}\sum_{1\leq k, j,j'\leq n}\sum_{-4\leq l,l'\leq 4} Z_k^{l}(\hat\theta_n)Z_k^{l'}(\hat\theta_n)\dot Z_j^{l}(\hat\theta_n)(\dot Z_{j'}^{l'}(\hat\theta_n))^\top$$
(see details in Section~\ref{sec:estimV}).
Thus we can estimate the covariance matrix $\Sigma$ and deduce an asymptotic confidence region. 

{ 

We also prove the following result on the quadratic risk of the estimator $\hat\theta_n $, which  is useful for the sequel (see Section~\ref{sec:proof1varthetachap}  for a proof).
\begin{proposition}\label{varthetachap}
Under the assumptions of Theorem~\ref{normasymp}, there exists a numerical constant $K$ such that, for all $\theta_0\in \Theta$ and for all $n\geq 1$ 
$$\E\|\hat\theta_n-\theta_0\|^2\leq Kn^{-1},$$
where the norm is the Euclidean norm in $\R^3$.
\end{proposition}
}

\section{Nonparametric part}

Let us now estimate the nonparametric part. We shall use the following norm: for any function $\phi$, we denote $\|\phi\|_2=\left(\frac{1}{2\pi}\int_{\S^1} \phi^2(x)dx\right)^{1/2}.$ 
Recall that for all $l\in \Z$, $g^{\star l }=M^l(\theta_0)f^{\star l}$
 where $g$ is the density of the observations $X_k$ and $g^{\star l }$ its Fourier coefficient.
Then $f^{\star l}=g^{\star l }/M^l(\theta_0)$. We can verify that $M^{l}(\theta_0)\neq 0$. Indeed, for any $\theta\in \Theta$, 
$$|M^l(\theta)|^2=p^2+(1-p)^2+2p(1-p)\cos[l(\beta-\alpha)]\geq (1-2p)^2>0.$$
Nevertheless this division by $M^l(\theta_0)$  requires us to impose a new assumption. We assume that there exists $P\in (0,1/2)$ such that $0<p<P$ for any $p$, i. e. 
\begin{assumption}\label{hypthetabis} $\Theta$ is a compact set included in 
$$\left\{(p, \alpha, \beta)\in \left(0, P\right)\times [0,\pi)\times [0,\pi), \quad \alpha\neq \beta  \pmod{2\pi/3}\right\}.$$
\end{assumption}
Under this assumption, $|M^l(\theta)|$ is always bounded from below  by $1-2P$.
Now, to estimate $g^{\star l }=\int_{\S^1}e^{-il x}g(x)dx/(2
\pi)$, it is natural to  define
$$\widehat{g^{\star l}}=\frac1{2\pi n}\sum_{k=1}^n e^{-ilX_k}.$$
If $\hat \theta=\hat \theta_n$ is the previous estimator of the parametric part, we set  the plugin estimator of the Fourier coefficient:
$$\widehat{f^{\star l}}=\frac1{2\pi n}\sum_{k=1}^n M^{l}(\hat \theta)^{-1}e^{-ilX_k}.$$
Finally, for 
{$L$ an integer}, set
$$\hat f_L(x)=\sum_{l=-L}^L\widehat{f^{\star l}}e^{ilx}. $$

To measure the performance of this estimator, we use Parseval equality to write
$$\|f-\hat f_L\|_2^2=\sum_{|l|>L} |f^{\star l}|^2+\sum_{l=-L}^L |f^{\star l}-\widehat{f^{\star l}}|^2$$
which is the classical bias variance decomposition.
Moreover it is possible to prove that the variance term satisfy $\sum_{l=-L}^L \E|f^{\star l}-\widehat{f^{\star l}}|^2=O(\frac{2L+1}{n})$ (see Lemma~\ref{majnun} below).
 To control the bias term we recall %
 {the definition of the 
Sobolev ellipsoid:
$$W(s, R)=\{ f: \S^1\to \R,\quad  \sum_{l\in \Z} (1+l^2)^{s} |f^{\star l}|^2 \leq R^2\}.$$ }

For such a smooth $f$, 
the risk of estimator $\hat f_L$ is then bounded in the following way:
$$\E\|f-\hat f_L\|_2^2\leq R^2\left(1+L^2\right)^{-s}+ C\frac{2L+1}{n}. $$
It is clear that an optimal value for $L$ is of order $n^{1/(2s+1)}$ but this value is unknown. We rather choose a data-driven method to select $L$. 
We introduce a classical minimization of a penalized empirical risk. Set
\begin{equation}\label{resolution_level}
\widehat{L}=\underset{L\in \mathcal{L}}{\argmin} \left\{-\sum_{l=-L}^L | \widehat{f^{\star l}}|^2+\lambda \frac{2L+1}{n}\right\}
\end{equation}
where $\mathcal{L}$ is a finite set of resolution level, and $\lambda$ a constant to be specified later.
The next theorem states an oracle inequality which highlights the bias variance decomposition of the quadratic risk and  justifies our estimation procedure.
\begin{theorem}\label{oracle}
\label{oracleineq}
Assume Assumption \ref{hypf} and Assumption \ref{hypthetabis} . 
Assume that $f$ belongs to the Sobolev ellipsoid $W(s, R)$ with 
{  $s \geq 1$. Let $\hat L$ defined  in \eqref{resolution_level} with 
$\mathcal{L}=\{{0},1, \dots, \lfloor cn^{\frac1{2s_0+1}} \rfloor \}$ for some $s_0>1$ and some positive constant $c$.}
Let $\epsilon>0$.
If the penalty constant verifies 
$\lambda > {(3/\pi^2)}(1+\epsilon^{-1})(1-2P)^{-2}$
then, 
$$\E\|\hat f_{\widehat{L}}-f\|_2^2\leq (1+2\epsilon)\E\min_{L\in \mathcal{L}} \left\{\|\hat f_{ L}-f\|_2^2+2\lambda \frac{2L+1}{n}\right\}+\frac{C {(1+R^2)}}{n}$$
where $C$ is  a positive constant depending on 
{ 
 $P,s_0,c,\epsilon,\lambda$.
 Moreover, if $s\geq s_0$,
$$\sup_{f\in W(s, R)}\sup_{\theta_0\in \Theta}\E_{\theta_0,f}\|\hat f_{\widehat{L}}-f\|_2^2=O\left(R^2
n ^{-2s/(2s+1)}\right).$$
}
\end{theorem}

In consequence our estimator has a quadratic risk in  $n ^{-2s/(2s+1)}$.
{ 
Regarding the lower bound note that for any estimator $\tilde f_n$
$$\sup_{f\in W(s, R)}\sup_{\theta_0\in \Theta}\E_{\theta_0,f}\| \tilde f_n-f\|_2^2\geq 
\sup_{f\in W(s, R)}\E_{\theta,f}\|\tilde f_n-f\|_2^2$$
for some arbitrary $\theta\in \Theta$,
so that the problem is reduced to a purely nonparametric lower bound.
In the case of direct observations  this quantity is lower bounded by $\underline{C}n ^{-2s/(2s+1)}$, see Theorem 11 and its proof in  \cite{baldikerk2009} (case $d=1$ for the circle $\S^1$).
We can use this proof to prove the lower bound in our mixture case. Indeed, for any densities $f_1$ and $f_2$, if 
$g_i(x)=pf_i(x-\alpha)+(1-p)f_i(x-\beta)$ is the associated density of our observations, then the Kullback-Leibler divergence verifies
$$K(g_1dx,g_2dx)\leq \int \frac{(g_1-g_2)^2}{g_2}\leq 2 \int \frac{(f_1-f_2)^2}{f_2}$$
and the rest of the proof is identical. Thus 
$$\sup_{f\in W(s, R)}\sup_{\theta_0\in \Theta}\E_{\theta_0,f}\|\hat f_{\widehat{L}}-f\|_2^2\geq \underline{C}
n ^{-2s/(2s+1)}$$
and our estimator is optimal minimax. 
}

\begin{remark} 
 Note that the penalty only depends on $P$ which is some safety margin around  1/2, that can be chosen by the statistician. For the practical choice of the penalty, see Section \ref{section:simus}.
\end{remark}

 Eventually, note that some densities may be \textit{supersmooth}, in the following sense:
$$\sum_{l\in \Z} \exp(2b|l|^r)|f^{\star l}|^2 \leq R^2. $$
In this case, the quadratic bias is bounded by $R^2\exp(-2bL^r)$ which gives the following fast rate of convergence:
$$\E\|\hat f_{\widehat{L}}-f\|_2^2=O\left(\frac{(\log n)^{1/r}}{n}\right).$$

\section{Numerical results}
\label{section:simus}

All computations are performed with Matlab software and the Optimization Toolbox.

We shall implement our statistical procedure to both estimate the parameter $\theta_0$ and the density $f$.   We consider three popular circular densities, namely the von Mises density, the wrapped Cauchy and the wrapped normal densities. We remind their expression (see \cite{LeyVerdebout}).
The von Mises density is given by:
$$
f_{VM}(x)=\frac{1}{2\pi I_{0}(\kappa) } e^{\kappa \cos(x-\mu)},
$$
with $\kappa \geq 0$, $I_0(\kappa)$ the modified Bessel function of the first kind and of order $0$.
The wrapped Cauchy distribution has density:
$$
f_{WC}(x)= \frac{1}{2\pi} \frac{1-\gamma^2}{1+ \gamma^2 -2 \gamma\cos(x-\mu)},
$$ 
with $0 \leq \gamma \leq 1.$
The wrapped normal density expression is:
$$
f_{WN}(x)=\frac{1}{\sigma\sqrt{2\pi}}\sum_{k=-\infty}^\infty e^{- \frac{(x-\mu+2k\pi)^2}{2\sigma^2}},
$$
$\sigma>0$. For more clarity, we set $\sigma^2=:-2 \log (\rho).$ Hence, we have $0\leq \rho\leq 1$.

All these densities are characterized by a concentration parameter $\kappa$, $\gamma$ or $\rho$ and a location parameter $\mu$. Remind that values $\kappa=0$, $\gamma=0$ and $\rho=0$ correspond to the uniform density on the circle. To meet symmetry assumptions of Theorem \ref{identifiability}, we consider in the sequel that the location parameter is set to $\mu=0$. 
\medskip

First, let us focus on the parametric part. We set $\theta_0= (p_0, \alpha_0, \beta_0)=(\frac 1 4, \frac \pi 8, \frac{2\pi}{3})$. Obtaining the estimate $\hat \theta_n$ of $\theta_0$ (see (\ref{estimator_theta})) requires to solve a nonlinear minimization problem. To this end, we resort to the function \textit{fmincon} of the Matlab Optimization toolbox. The function \textit{fmincon}  finds a constrained minimum of a function of several variables. Two parameters are to be specified: the domain over which the minimum is searched and an initial value. We consider the domain $\{ (0,\frac 1 2) \times [0, \pi) \times [0, \pi) \}$. For more stability and to avoid possible local minimums, we perform the procedure over $10$ initials values uniformly drawn on  $\{ (0,\frac 1 2) \times [0, \pi) \times [0, \pi) \}$. The final estimator $\hat \theta_n$ corresponds to the minimum value of the empirical contrast $S_n(\theta)$ given in (\ref{Sn}) over the $10$ trials.

Table \ref{MISE-theta} gathers mean squared errors for our estimation procedure. When analyzing Table \ref{MISE-theta}, one clearly sees that increasing the number of observations improves noticeably the performances. As expected, von Mises densities with smaller concentration parameter are more difficult to estimate. Nonetheless, the overall performances are satisfying. Table \ref{MISE-theta-Spurr} displays the performances of the method-of-moments estimation procedure developed by \cite{SpurrKout} to handle the problem of estimating the parameters in mixtures of von Mises distributions. 
To fairly compare the two methods, Table \ref{MISE-theta-Spurr-2} gives the \cite{SpurrKout} performances but this time when estimating on the same domain than ours e.g $\{ (0,\frac 1 2) \times [0, \pi) \times [0, \pi) \}$. At closer inspection, the \cite{SpurrKout} method seems to behave better to estimate angles $\alpha_0$ and $\beta_0$ while our method may appear more competitive for estimating $p_0$. It is worth noticing that the method by \cite{SpurrKout} is completely parametric and takes advantage of the knowledge of the distributions. In this regard, our procedure which is semiparametric is competitive with a parametric method.

\medskip

Figure \ref{histogrammes} illustrates the asymptotic normality of our estimator $\hat \theta_n$ stated in Theorem \ref{normasymp}.

\bigskip

\begin{table}[!h]  \begin{center}  \begin{tabular}{c |c c c| c c c}
density  & \multicolumn{3}{c}{$n=100$ } & \multicolumn{3}{c}{$n=1000$}\\
&  $p$  &  $\alpha$  &  $\beta$ &  $p$ &  $\alpha$ &  $\beta$ \\
\hline
$f_{VM}$, $\kappa=2$  & 0.0121    & 0.6848   &   0.1131    & 0.0017   &  0.1919  &  0.0238 \\
$f_{VM}$, $\kappa=5$     &  0.0030&   0.0285 &  0.0049  &  1.4632e-04 & 0.0017 & 4.4861e-04  \\
$f_{VM}$, $\kappa=7$     & 0.0033  & 0.0133  & 0.0031  & 1.6721e-04  & 0.0013 & 3.0102e-04 \\
$f_{WC}$, $\rho=0.8$ & 0.0029 & 0.0124 &0.0024  & 2.0788e-04 & 8.5435e-04  &  1.8942e-04  \\ 
$f_{WN}$, $\rho=0.8$  &0.0077  &  0.1679  &  0.0457& 0.0020 & 0.0238&  0.0037 
\end{tabular}
\caption{Mean squarred errors for estimating parameter $\theta_0$ over $50$ Monte Carlo replications.}
\label{MISE-theta}
\end{center}\end{table}

\begin{table}[!h]  \begin{center}  \begin{tabular}{c |c c c| c c c}
density  & \multicolumn{3}{c}{$n=100$ } & \multicolumn{3}{c}{$n=1000$}\\
&  $p$  &  $\alpha$  &  $\beta$ &  $p$ &  $\alpha$ &  $\beta$ \\
\hline
$f_{VM}$, $\kappa=2$  &   0.0938   & 0.4212  &   0.1171     & 0.0116   &   0.0685 & 0.0062  \\
$f_{VM}$, $\kappa=5$     & 0.0031   & 0.0360    & 0.0049  & 2.9965e-04  & 0.0025  & 6.6273e-04   \\
$f_{VM}$, $\kappa=7$     &  0.0031   &0.0084  &  0.0029  & 2.4553e-04  &  0.0014  &3.5541e-04   \\
\end{tabular}
\caption{Spurr and Koutbeiy procedure: mean squared errors for estimating parameter $\theta_0$ over $50$ Monte Carlo replications on $\{ (0,1) \times [0, 2\pi) \times [0, 2\pi) \}$}
\label{MISE-theta-Spurr}
\end{center}\end{table}

\begin{table}[!h]  \begin{center}  \begin{tabular}{c |c c c| c c c}
density  & \multicolumn{3}{c}{$n=100$ } & \multicolumn{3}{c}{$n=1000$}\\
&  $p$  &  $\alpha$  &  $\beta$ &  $p$ &  $\alpha$ &  $\beta$ \\
\hline
$f_{VM}$, $\kappa=2$  & 0.0231    &  0.2117  &   0.0351     & 0.0112  &0.0635    & 0.0081   \\
$f_{VM}$, $\kappa=5$     & 0.0032   &  0.0409  & 0.0042  &  4.1489e-04 & 0.0022  & 6.3122e-04  \\
$f_{VM}$, $\kappa=7$     &  0.0026   &  0.0094 & 0.0029   & 2.3197e-04  & 0.0010   & 2.8350e-04  \\
\end{tabular}
\caption{Spurr and Koutbeiy procedure: mean squared errors for estimating parameter $\theta_0$ over $50$ Monte Carlo replications on $\{ (0,\frac 1 2) \times [0, \pi) \times [0, \pi) \}$}
\label{MISE-theta-Spurr-2}
\end{center}\end{table}

\begin{figure}[!h]\begin{center}
\begin{tabular}{cc}
$\alpha$ & $\beta$ \\ 
\includegraphics[scale=0.23]{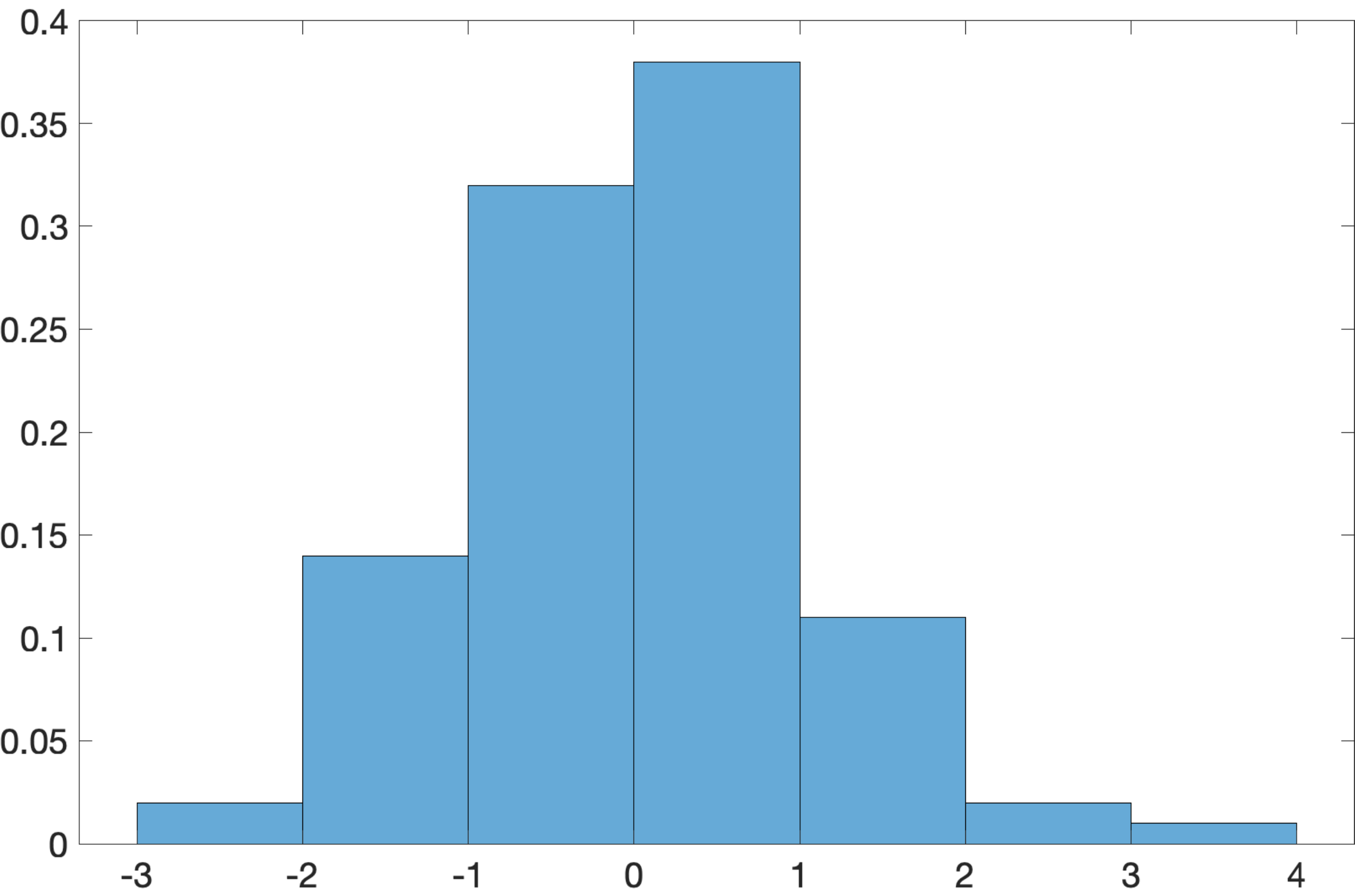} & \includegraphics[scale=0.23]{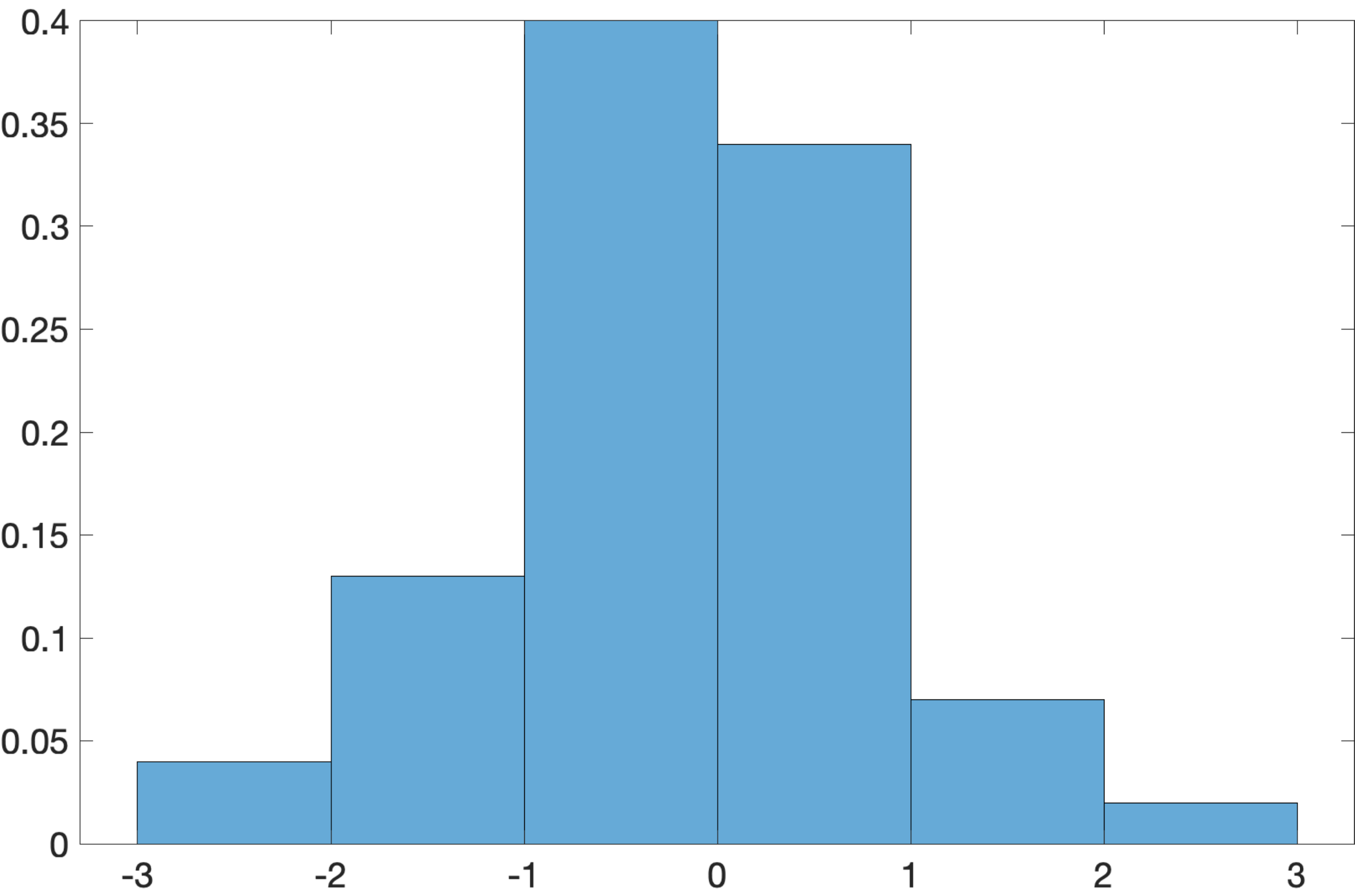}
 \end{tabular}
 $p$\\
 \includegraphics[scale=0.23]{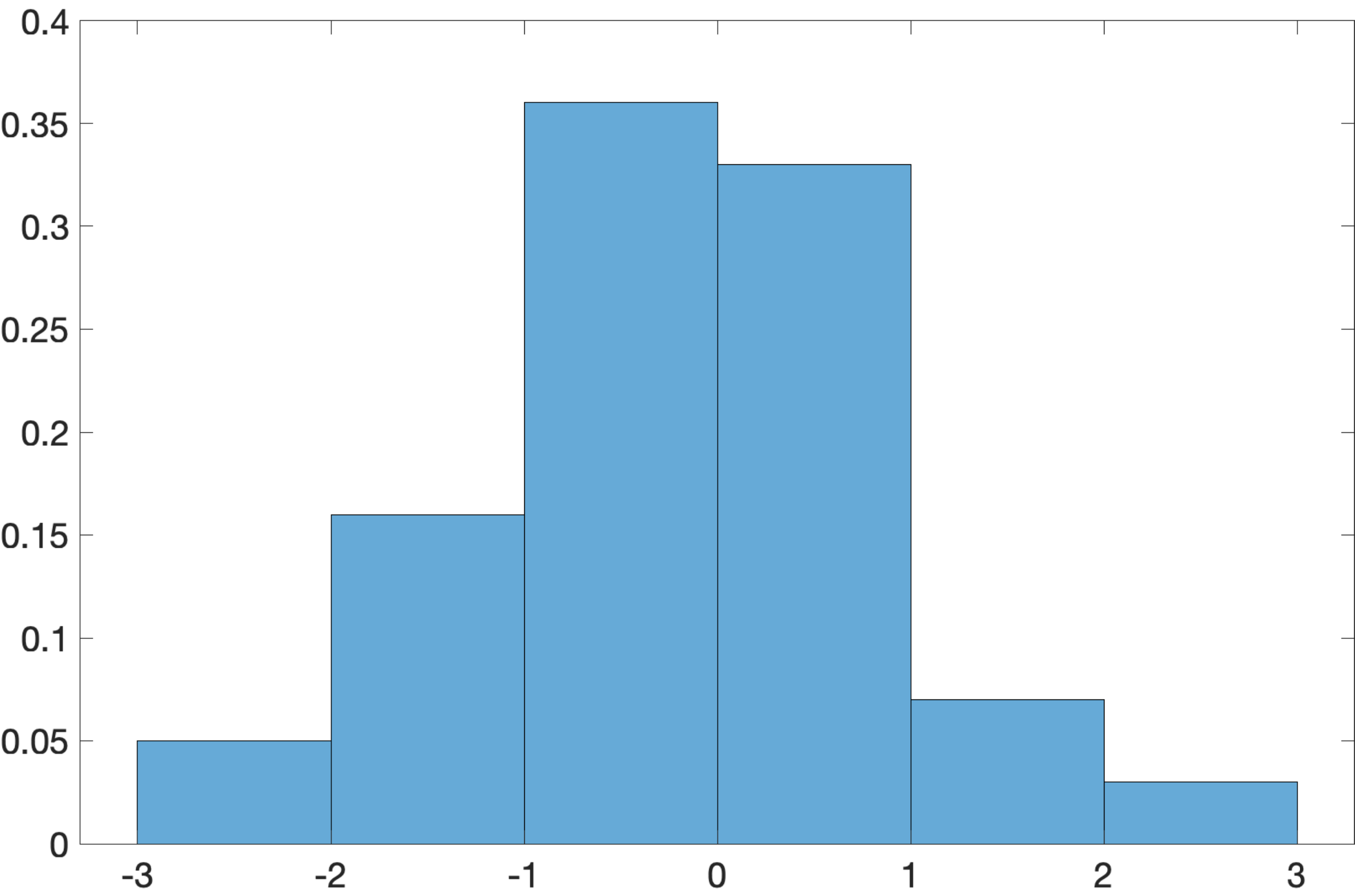}
\caption{Histograms of the centered and standardized statistics $\hat \theta_n$  for the von Mises density $f_{VM}$ with $\kappa=5$, $n=1000$ observations and $100$ Monte Carlo replications}
 \label{histogrammes}
\end{center}\end{figure}

Now, let us turn to the nonparametric estimation part namely the estimation of the density $f$. The estimator of $f$ is given by $\hat f_{\widehat{L}}$ (see Theorem \ref{oracleineq}). It requires the computation of a data-driven resolution level  choice $\widehat{L}$  (given in (\ref{resolution_level})) which implies a tuning parameter $\lambda$. To select the proper $\lambda$, we follow the data-driven slope estimation approach due to Birg\'e and Massart (see \cite{BirgeMassart2001} and \cite{BirgeMassart2007}). An overview in practice is presented in \cite{Baudry}. To implement the slope heuristics method, one has to plot for ${L=0}$ to $L_{\max}$ the couples of points $(\frac{2L+1}{n} , \sum_{l=-L}^L | \widehat{f^{\star l}}|^2)$.  For $L \geq L_0$, one should observe  a linear behaviour (see Figure \ref{slope}). Then, once the slope is estimated,  say $a$, by a linear regression method, one eventually takes $\hat \lambda = 2a$ and the final resolution level is:
$$\widehat{L}=\underset{L\in \mathcal{L}}{\argmin} \left\{-\sum_{l=-L}^L | \widehat{f^{\star l}}|^2+\hat \lambda \frac{2L+1}{n}\right\}.$$

\begin{figure}[!h]\begin{center}
\begin{tabular}{c}
\includegraphics[scale=0.35]{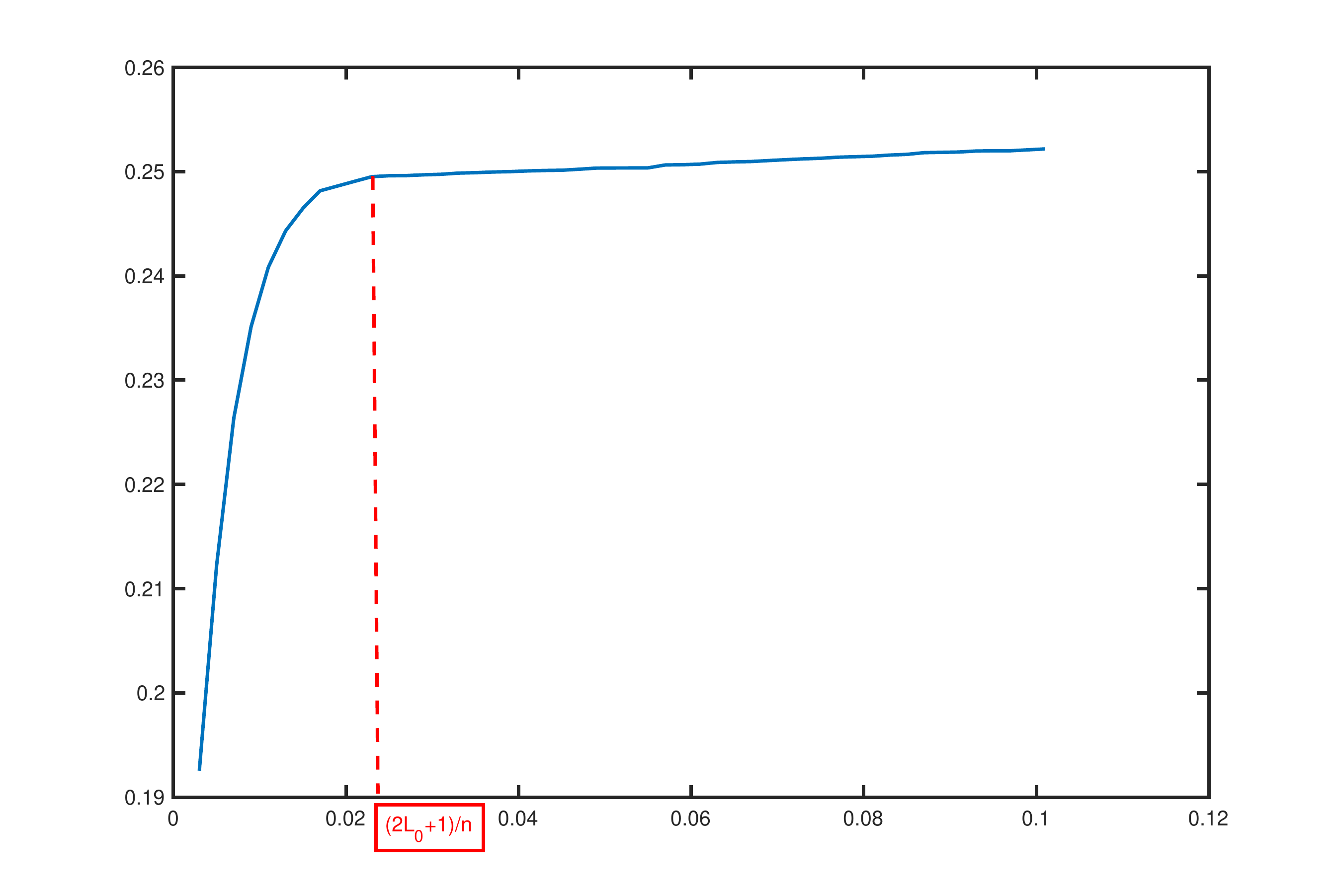} 
\end{tabular}
\caption{For the wrapped Cauchy density $f_{WC}$ with $\gamma=0.8$ and $n=1000$: plot of couples $(\frac{2L+1}{n}, \sum_{l=-L}^L | \widehat{f^{\star l}}|^2)$  for $L=\{1,\dots, 50 \} $.}
 \label{slope}
\end{center}\end{figure}

\begin{figure}[!h]\begin{center}
\begin{tabular}{cc}
 $f$ &  $g$ \\ 
\includegraphics[scale=0.23]{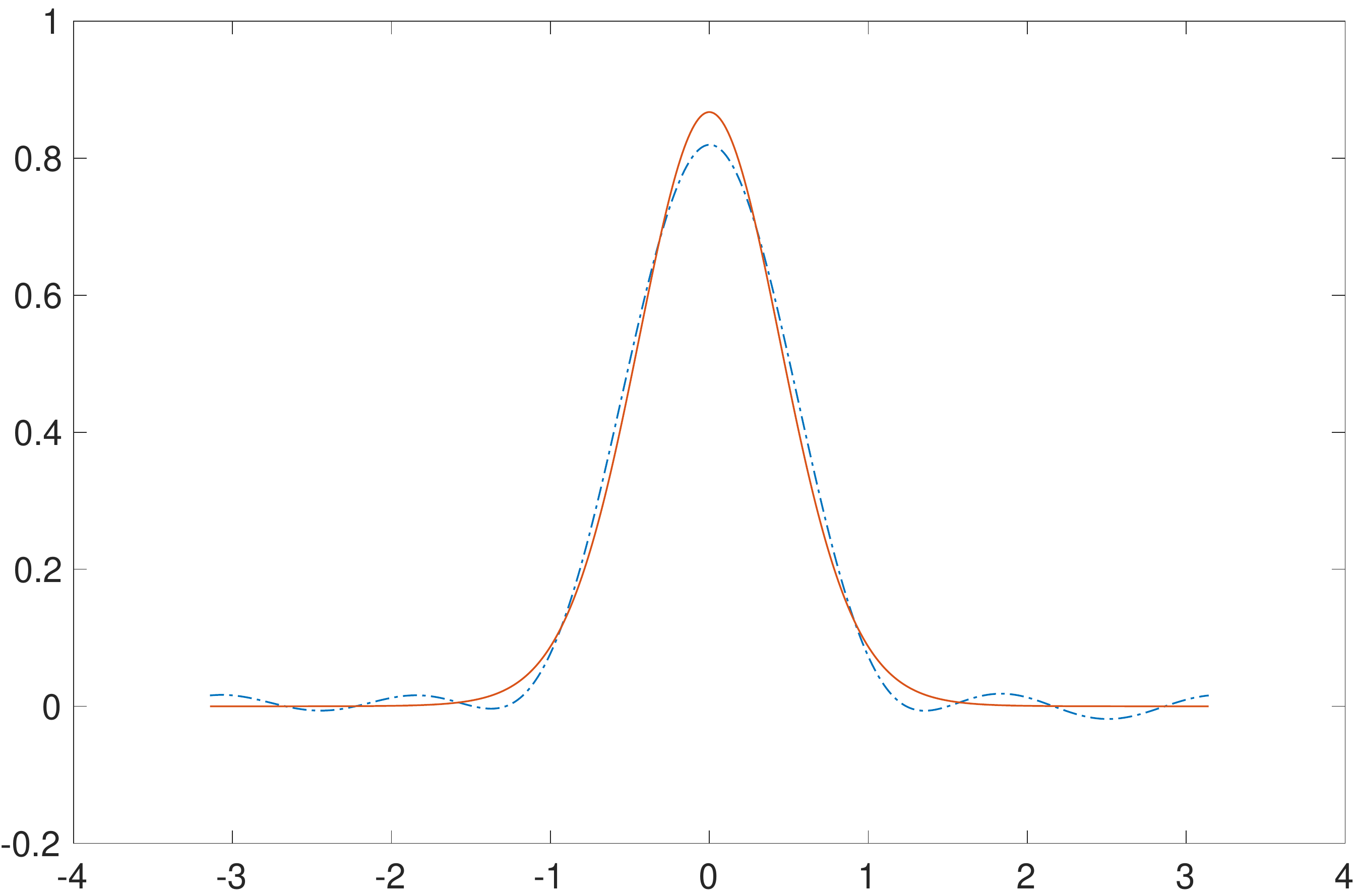}&\includegraphics[scale=0.23]{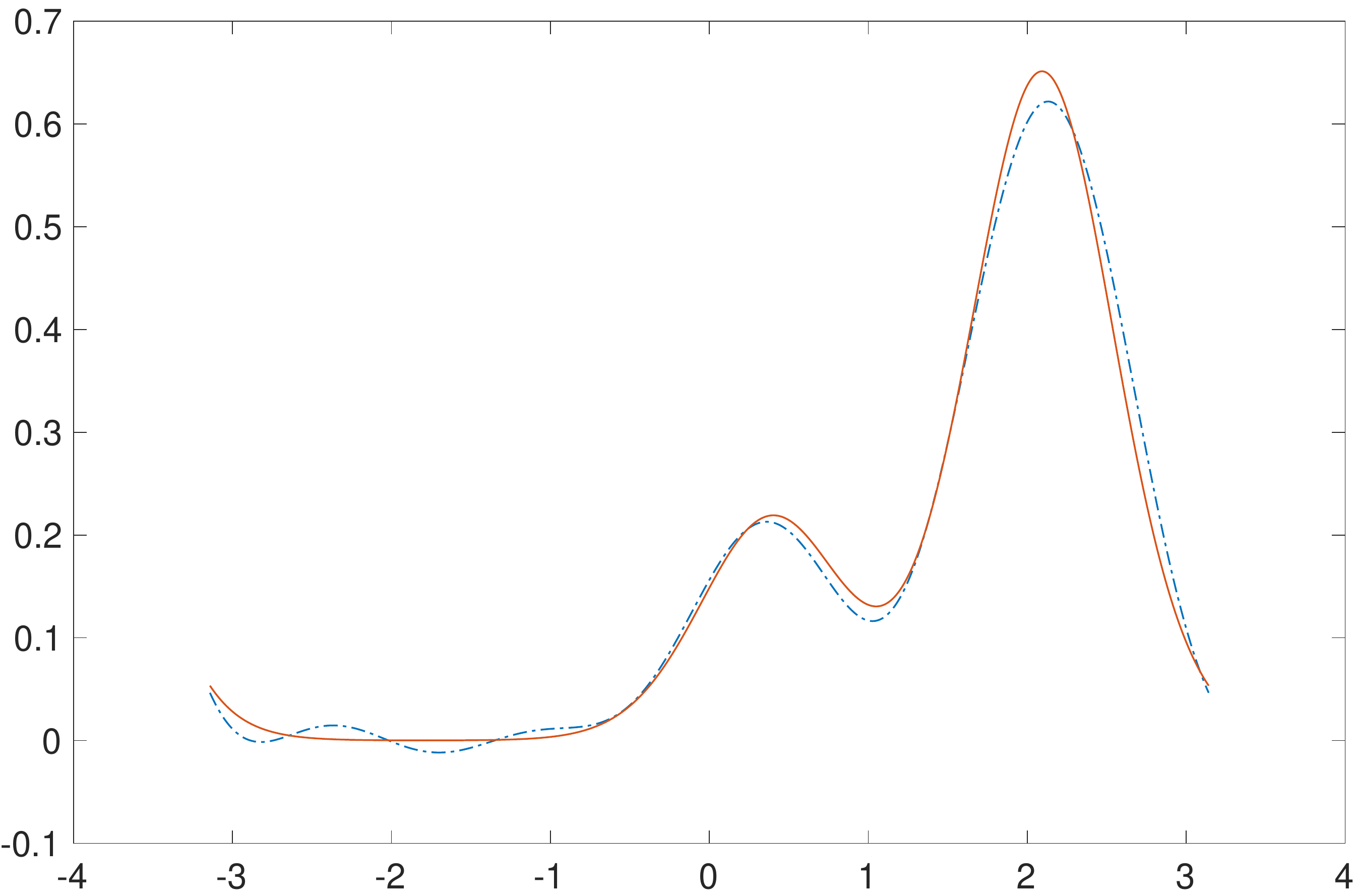} \\
\includegraphics[scale=0.23]{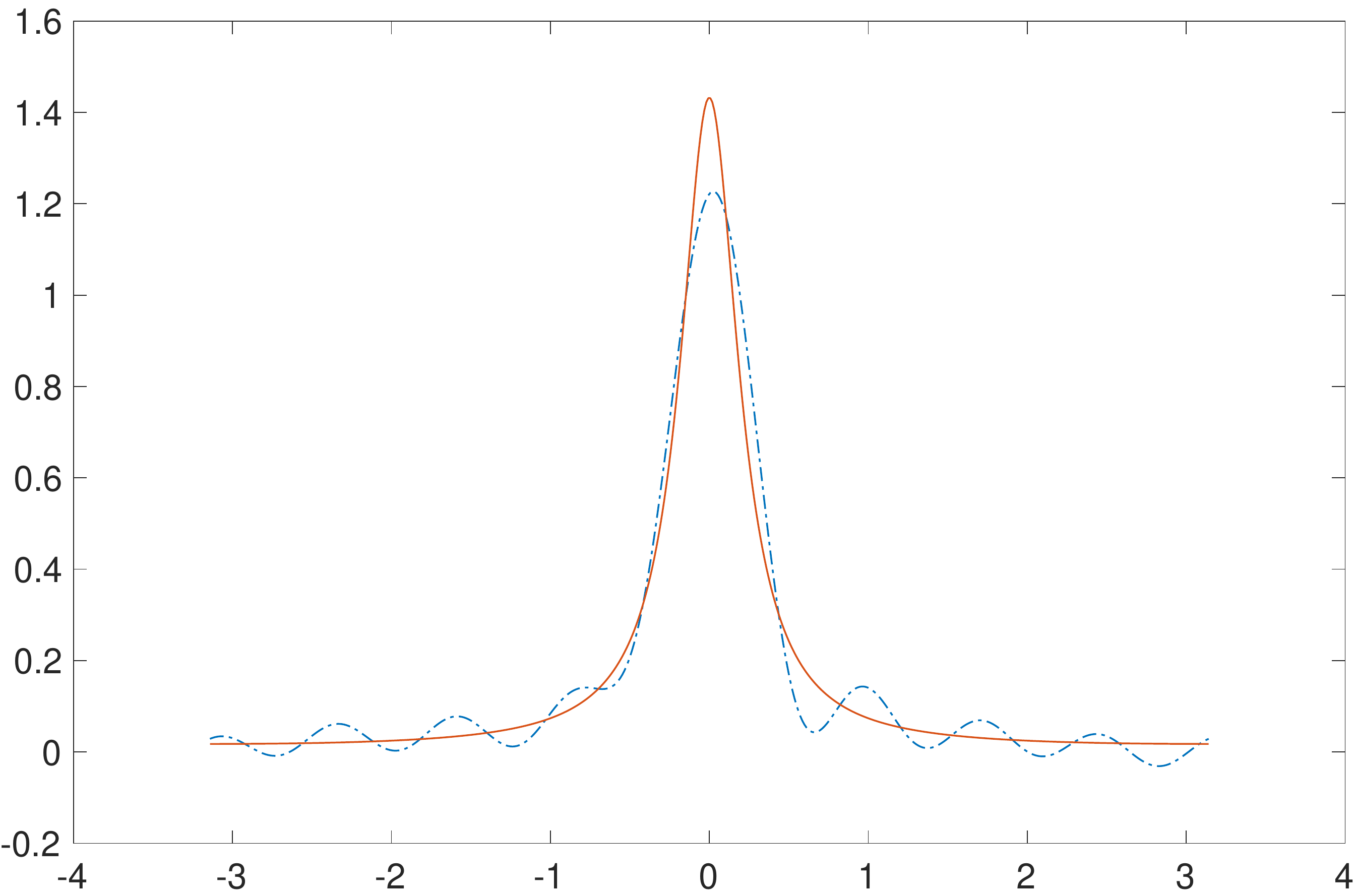}&\includegraphics[scale=0.23]{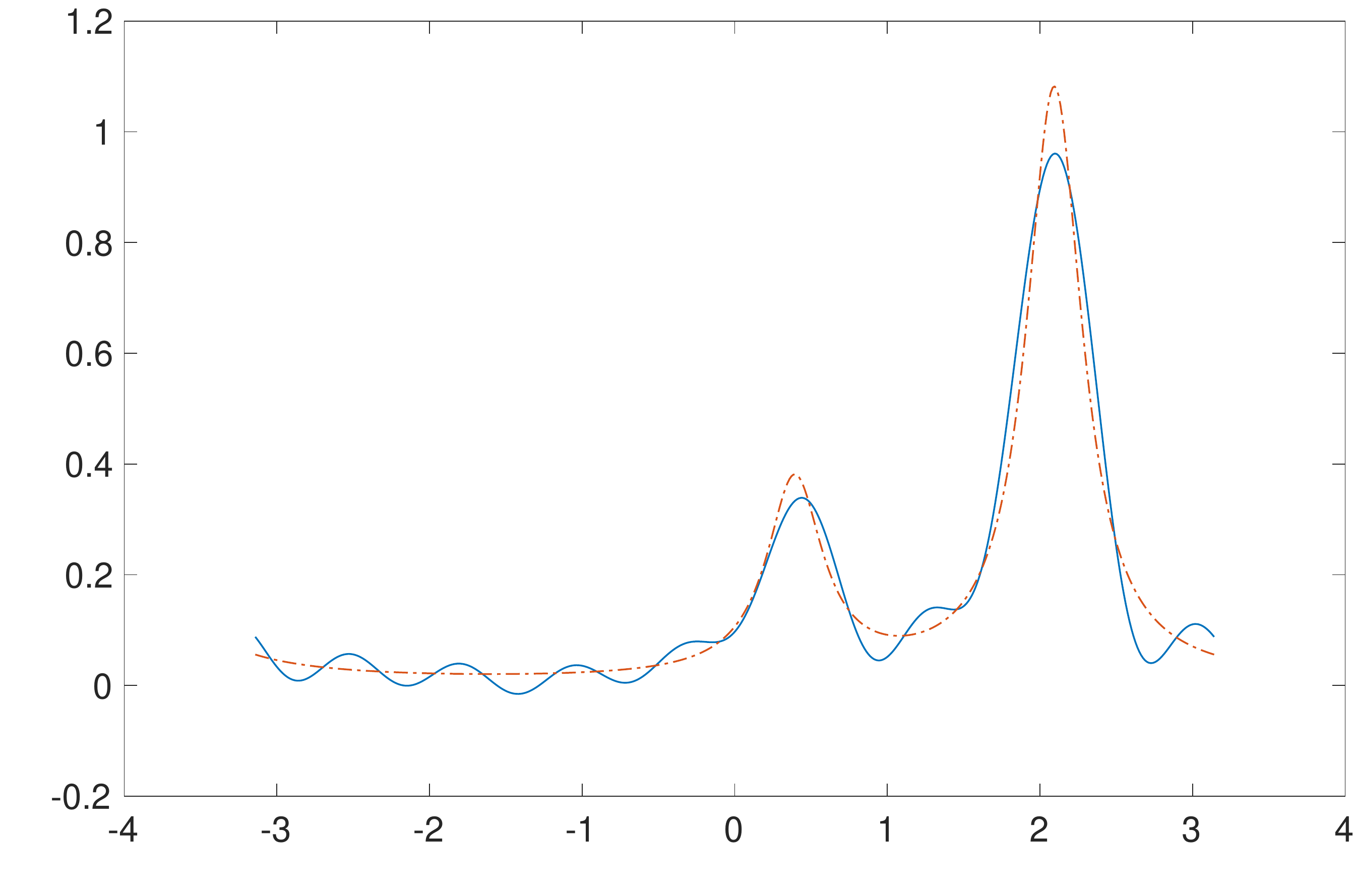}  \\
\includegraphics[scale=0.23]{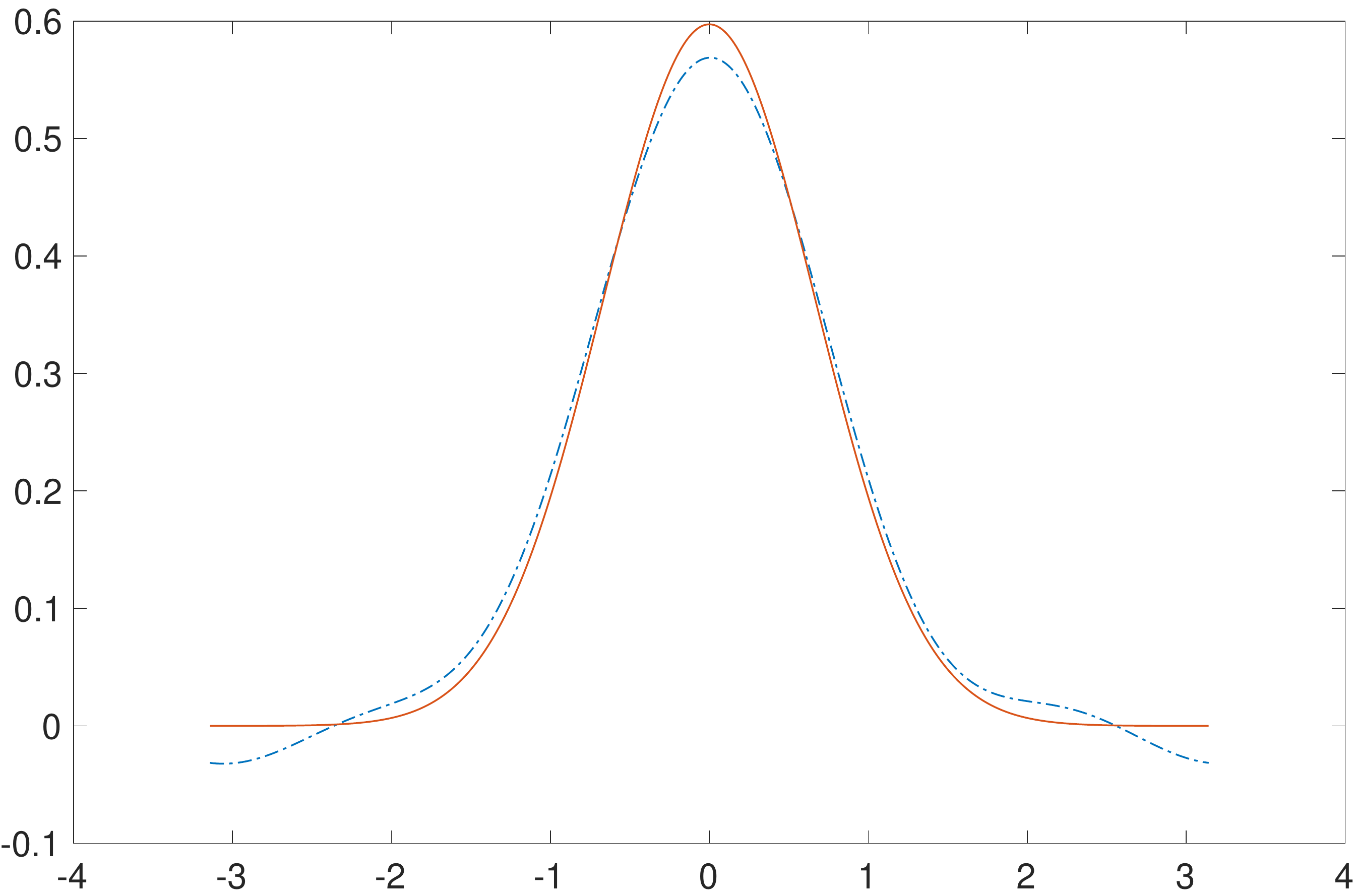}&\includegraphics[scale=0.23]{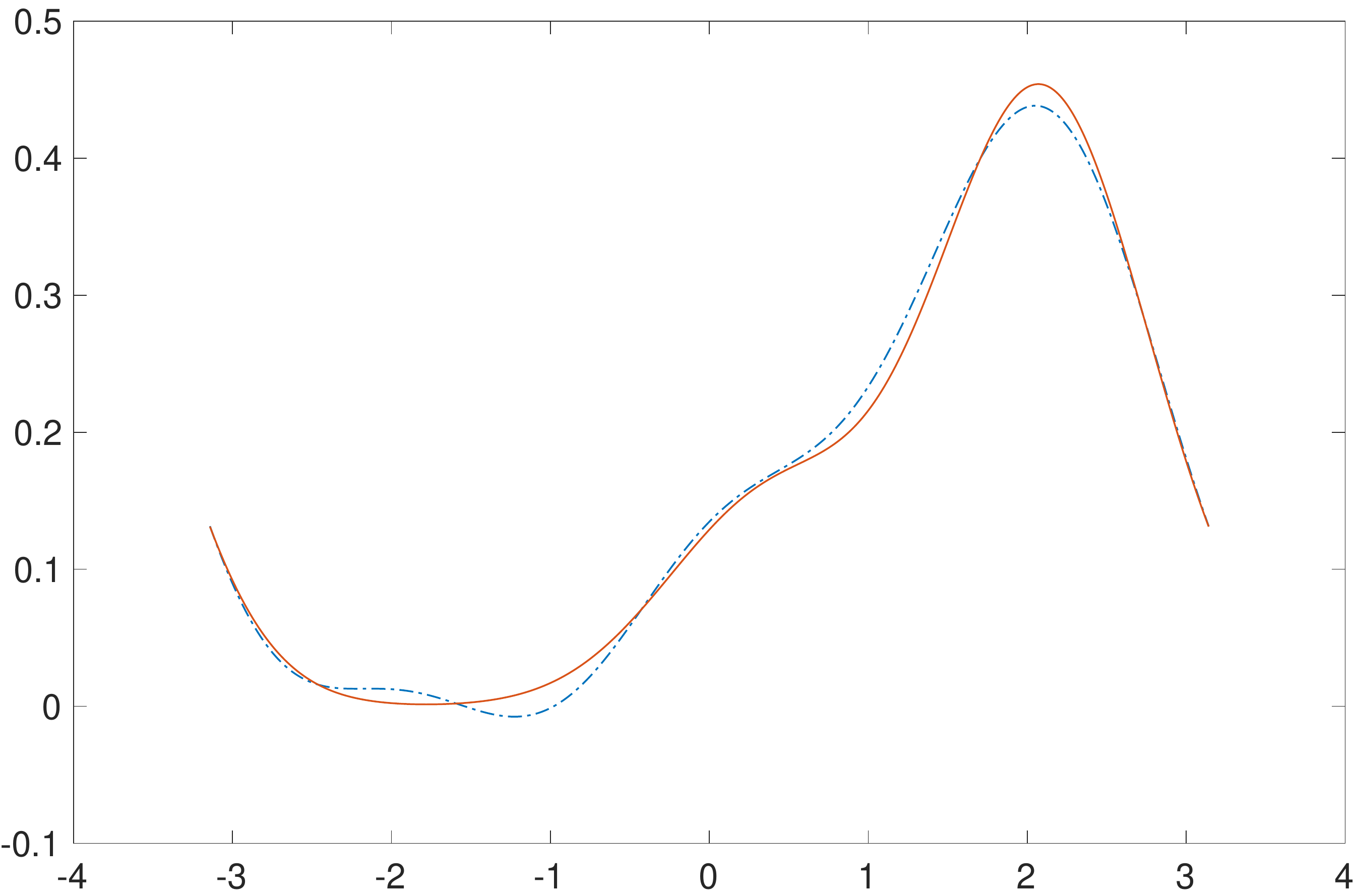}  
\end{tabular}
\caption{Estimation of the density $f$ and the mixture density $g$ for $n=1000$. In red, the density, in dotted lines its estimate. From top to bottom:  the von Mises density with $\kappa=5$, the wrapped Cauchy with $\gamma=0.8$ and the wrapped normal density with $\rho=0.8$. }
 \label{graphiques}
\end{center}\end{figure}

Finally, Figure \ref{graphiques} shows  reconstructions of the density $f$ and the mixture density $g$ as well. The estimates are  good.  
\begin{remark}
Note that for the two exceptional cases, when $p_0=0$ or $f$ is the uniform density, our procedure performs well. Indeed, if $p_0=0$, our method yields that $\alpha=\beta$ and retrieves that there is only one component in the mixture. When $f$ is the uniform density, 
our algorithm selects $\hat L=0$ which yields the uniform distribution.
\end{remark}

\section{Proofs}

\subsection{Proof of Theorem~\ref{identifiability} (identifiability)}

Denote 
\[
M^l(\theta):= pe^{-i \alpha l }+ (1-p)e^{-i\beta l}.
\]
Suppose 
$p f( x-\alpha)+(1-p) f(x-\beta)=p' f'( x-\alpha')+(1-p') f'(x-\beta')$.
The calculation of  the Fourier coefficients gives, for all $l\in \Z$, 
$ f^{\star l}M^l(\theta)=(f')^{\star l} M^l(\theta')$
 which implies $$ f^{\star l}|M^l(\theta)|^2=(f')^{\star l} M^l(\theta')\overline{M^l(\theta)}.$$ Then, our assumptions on $f$ and $f'$ entail 
 $$M^l(\theta')\overline{M^l(\theta)}\text{ is real}\quad \forall l\in \{1,2,3,4\}.$$

Let us now study the consequence of this fact.
Denote 
$$\gamma_1=\alpha'-\beta,\; \gamma_2=\alpha'-\alpha,\;\gamma_3=\beta'-\beta,\;\gamma_4=\beta'-\alpha$$
the 4 angles. 
Denote also the associated weights in $(0,1)$:
$$\lambda_1=p'(1-p),\; \lambda_2=p'p, \;\lambda_3=(1-p')(1-p),\; \lambda_4=(1-p')p.$$
With this notation 
$$
M^l(\theta')\overline{M^l(\theta)}=
\lambda_1e^{-i\gamma_1l}+\lambda_2e^{-i \gamma_2 l}+\lambda_3e^{-i \gamma_3 l} +\lambda_4e^{-i\gamma_4l}.$$
Then  
$M^l(\theta')\overline{M^l(\theta)}$ is real if and only if $\sum_{k=1}^4 \lambda_k \sin (l\gamma_k)=0$
and we have to solve the equations 
 \begin{equation}
\label{equasin}
\forall l=1,2,3,4,\qquad \sum_{k=1}^4 \lambda_k \sin (l\gamma_k)=0.
\end{equation}
This system of equations  is studied in Lemmas~\ref{detA} and \ref{CasAnnul}
below.

Let us now reason with the representatives of the $\gamma_k$ in $(-\pi,\pi]$. Lemma~\ref{CasAnnul} says that the possible values for the $\gamma_k$'s are  $0, \pi, \gamma, -\gamma$, for some $\gamma \in (0,\pi)$. 
Note that here
\begin{equation}
\label{eqdif}\gamma_1-\gamma_2=\gamma_3-\gamma_4=\alpha-\beta\neq 0\;\;
\text{ and }\;\;
\gamma_1-\gamma_3=\gamma_2-\gamma_4=\alpha'-\beta'\neq 0
\end{equation}
and then 
the $\gamma_k$'s take at least 2 different values: either 4 different values; or $\gamma_2=\gamma_3$ and the other distinct; or $\gamma_1=\gamma_4$ and the other distinct; or $\gamma_2=\gamma_3$ and $\gamma_1=\gamma_4$. 

$\bullet$ Let us first study the case where all the $\gamma_k$'s are distinct. There are 4!=24 ways of having 
$(\gamma_{i_1}, \gamma_{i_2}, \gamma_{i_3}, \gamma_{i_4})=(-\gamma, 0, \gamma,\pi)$. But  16 combinations lead to  $p=1/2$ or $p'=1/2$.
For example, if 
$(\gamma_{1}, \gamma_{2}, \gamma_{3}, \gamma_{4})=(-\gamma, 0, \gamma, \pi)$ then \eqref{equasin} becomes 
$$\lambda_1 \sin (-l\gamma)+\lambda_2 \sin (0)+\lambda_3 \sin (l\gamma)+\lambda_4 \sin (l\pi)=0.$$
Thus $\lambda_1=\lambda_3$, which gives $p'=1/2$. In the same way, there are 4 possibilities giving $\lambda_1=\lambda_3$, 4 possibilities giving $\lambda_1=\lambda_2$, 4 possibilities giving $\lambda_2=\lambda_4$, 4 possibilities giving $\lambda_3=\lambda_4$. All of this is impossible, since $p,p'\in (0,1)\backslash\{ 1/2\}$.
In addition, in the 4 cases where  
$\gamma_1=-\gamma_4$, we obtain via \eqref{eqdif} $\gamma_3=-\gamma_2$ which is impossible if $\{\gamma_2,\gamma_3\}=\{0,\pi\}$.
Idem if $\gamma_2=-\gamma_3$ and  $\{\gamma_1,\gamma_4\}=\{0,\pi\}$. Thus it is finally impossible that all the $\gamma_k$'s are distinct. 

$\bullet$ Let us now study the case where the $\gamma_k$'s take 3 distinct values ($\gamma_2=\gamma_3$ or $\gamma_1=\gamma_4$) and belong to $\{0, \pi, \gamma\}$ or $\{0, \pi, -\gamma\}$. In the case where $\gamma_2=\gamma_3$, coming back to equation $\eqref{equasin}$, we understand that all the rearrangements lead to $\lambda_4=0$ or $\lambda_1=0$ or $\lambda_2+\lambda_3=0$, which is impossible. In the same way, if $\gamma_1=\gamma_4$, equation \eqref{equasin} leads to 
$\lambda_2=0$ or $\lambda_3=0$ or $\lambda_1+\lambda_4=0$, which is impossible.

$\bullet$ The next case is when the $\gamma_k$'s take 3 distinct values and belong to $\{0, \gamma, -\gamma\}$ or $\{\pi, \gamma, -\gamma\}$. If $\gamma_2=\gamma_3$, we can then list the 6 cases: 
$$\begin{array}{ccc|c}
\gamma_1 & \gamma_2=\gamma_3 & \gamma_4 & \text{consequence}\\
\hline
-\gamma & 0/\pi & \gamma &  p=p',
{ \alpha'-\alpha=\beta'-\beta =0\pmod{\pi}} \\
\gamma & 0/\pi &- \gamma & p=p',{ \alpha'-\alpha=\beta'-\beta =0\pmod{\pi}} \\
-\gamma & \gamma  &  0/\pi &\lambda_1=\lambda_2+\lambda_3
\\
\gamma & -\gamma  &  0/\pi &\lambda_1=\lambda_2+\lambda_3
\\
0/\pi & \gamma & -\gamma& \lambda_4=\lambda_2+\lambda_3
\\
0/\pi & -\gamma & \gamma &\lambda_4=\lambda_2+\lambda_3
\\
\end{array}
$$
Note that $\lambda_1=\lambda_2+\lambda_3\Leftrightarrow p'(2-3p)=1-p$,
which is possible only if $p<1/2$ and $p'>1/2$ (recall that we suppose $p<1/2$ and $p'<1/2$). In the same way $\lambda_4=\lambda_2+\lambda_3\Leftrightarrow p'(1-3p)=1-2p$, which is possible only if $p>1/2$ and $p'<1/2$.\\
Finally, if  $\gamma_1=\gamma_4$, we have the 6 last cases:
$$\begin{array}{ccc|c}
\gamma_2 & \gamma_1=\gamma_4 & \gamma_3 & \text{consequence}\\
\hline
-\gamma & 0/\pi & \gamma &  p'=1-p\\
\gamma & 0/\pi &- \gamma & p'=1-p,\\
-\gamma & \gamma  &  0/\pi & p'=\frac{p}{3p-1}
\\
\gamma & -\gamma  &  0/\pi &  p'=\frac{p}{3p-1}
\\
0/\pi & \gamma & -\gamma& 
%
{p'=\frac{1-2p}{2-3p},  \beta- \alpha=\pm 2\pi/3}\\
0/\pi & -\gamma & \gamma&
%
{p'=\frac{1-2p}{2-3p},  \beta- \alpha=\pm 2\pi/3} 
\end{array}
$$
Note that the 4 first lines of this table are impossible since $p,p'\in(0,1/2)$ and  $p'=p/(3p-1)\notin (0,1)$ if $0<p<1/2$. Let us detail the lines 5 and 6. In these cases, $\lambda_1+\lambda_4-\lambda_3=0$ which provides $p'=(1-2p)/(3-2p)$. 
Moreover \eqref{eqdif} implies that $3\gamma_1=\gamma_2=0\pmod{\pi}$ and $2\gamma_1=\beta-\alpha=\alpha'-\beta'$.
According to the values of $\gamma_1$ and $\gamma_2$, there are 4 possibilities

$\diamond$ $\beta-\alpha= 2\pi/3$ and $(\alpha',\beta')=(\alpha, \beta +2\pi/3)$, 

$\diamond$ $\beta-\alpha= 2\pi/3$ and $(\alpha',\beta')=(\alpha+\pi, \beta -\pi/3)$

$\diamond$ $\beta-\alpha= -2\pi/3$ and $(\alpha',\beta')=(\alpha, \beta -2\pi/3)$

$\diamond$ $\beta-\alpha=- 2\pi/3$ and $(\alpha',\beta')=(\alpha+\pi, \beta +\pi/3)$

\color{black}

$\bullet$ The last case occurs when the $\gamma_k$'s take 2 distinct values.
If the $\gamma_k$'s take exactly  2 different values, using \eqref{eqdif}, necessarily
$$\gamma_1=\gamma_4 \text{ and }\gamma_2=\gamma_3 \pmod{2\pi}
\Rightarrow 0=\gamma_1-\gamma_4 +\gamma_3-\gamma_2=2(\alpha-\beta) \pmod{2\pi}$$
which is possible only if $\alpha-\beta= \pi\pmod{2\pi}$ (recall that $\alpha-\beta$ is always assumed $\neq 0$).  And in the same way $\alpha'-\beta'= \pi\pmod{2\pi}$.
Then $\gamma_1-\gamma_2=\alpha-\beta=\pi\pmod{2\pi}$. Thus the two different values of the $\gamma_k$'s are at a distant of $\pi$.

The first possibility is that these two values are $0$ and $\pi$, which corresponds to  the first case of Lemma~\ref{CasAnnul}. There are two subcases: 
1a. $(\gamma_1,\gamma_2,\gamma_3,\gamma_4)=(\pi,0,0,\pi)$ or 1b.
$(\gamma_1,\gamma_2,\gamma_3,\gamma_4)=(0,\pi,\pi,0)$.
In the subcase 1a. $(\alpha',\beta')=(\alpha, \beta)$. 
Equations
$$\begin{cases}
pf+(1-p)f_{\pi}=p'f'+(1-p')f'_{\pi}\\
pf_{\pi}+(1-p)f=p'f'_{\pi}+(1-p')f'
\end{cases} $$
entails that $f'$ is a linear combination of $f$ and $f_{\pi}$.
In the subcase $1b$.  $(\alpha',\beta')=(\alpha+\pi, \beta+\pi)=(\beta,\alpha)$.

The second possibility is that the two distinct  values $\gamma_1=\gamma_4$ and $\gamma_2=\gamma_3$ are not multiples of $\pi$, which 
corresponds to  the fourth case of Lemma~\ref{CasAnnul}. Then 
$(\gamma_1,\gamma_2,\gamma_3,\gamma_4)=(\gamma_1,-\gamma_1,-\gamma_1,\gamma_1)$ and $$\gamma_1-(-\gamma_1)=\gamma_1-\gamma_2=\pi \pmod{2\pi}$$ which entails $\gamma_1=\pi/2 \pmod{\pi}$.
Equation \eqref{equasin} becomes 
$$(\lambda_1-\lambda_2-\lambda_3+\lambda_4)\sin (l\pi/2)=0$$
so that $\lambda_1+\lambda_4=\lambda_2+\lambda_3$, which gives 
$$p'(1-p)+p(1-p')=p'p+(1-p')(1-p)\Rightarrow p'+p -2pp'=1/2\Rightarrow p'=1/2$$
which is impossible.

$\bullet$ Let us recap the only possible cases that we have obtained:

$\triangleright$ $p=p', \alpha'-\alpha=\beta'-\beta =0\pmod{\pi} $,
  
$\triangleright$ $p'=\frac{1-2p}{2-3p},  \beta- \alpha =\pm 2\pi/3$,
with the four possibilities described above,

$\triangleright$ $\beta-\alpha= \pi$, $(\alpha',\beta')=(\alpha, \beta)$ or $(\alpha',\beta')=( \beta, \alpha)$. 

This completes the proof of  the theorem.
 
\begin{lemma}\label{detA}
Let $\gamma_1, \dots, \gamma_4$ be four reals. Let $A$ be the matrix $(\sin(i \gamma_j))_{1\leq i,j\leq 4}$. Then 
$$\det A=64\prod_{k=1}^4\sin (\gamma_k) \prod_{1\leq i<j \leq 4} (\cos(\gamma_i)-\cos(\gamma_j)).
$$
\end{lemma}

\begin{proof}
From matrix $A$, doing line modification $L_3 \leftarrow L_3-L_1$, and $L_4 \leftarrow L_4-L_2$, we obtain (recall that 
$\sin(2p)=2\sin(p)\cos(p)$ and $\sin(p)-\sin(q)=2 \sin(\frac{p-q}{2})\cos(\frac{p+q}{2})$)

$$\det A= \begin{vmatrix}
\sin(\gamma_1) & \sin(\gamma_2) & \sin(\gamma_3) & \sin(\gamma_4) \\ 
2\sin(\gamma_1)\cos(\gamma_1) & 2\sin(\gamma_2)\cos(\gamma_2)& 2\sin(\gamma_3)\cos(\gamma_3)&2\sin(\gamma_4)\cos(\gamma_4) \\ 
2\sin(\gamma_1)\cos(2\gamma_1) & 2\sin(\gamma_2)\cos(2\gamma_2)& 2\sin(\gamma_3)\cos(2\gamma_3)&2\sin(\gamma_4)\cos(2\gamma_4) \\ 
2\sin(\gamma_1)\cos(3\gamma_1) & 2\sin(\gamma_2)\cos(3\gamma_2)& 2\sin(\gamma_3)\cos(3\gamma_3)&2\sin(\gamma_4)\cos(3\gamma_4) 
\end{vmatrix}. $$
Using $4$-linearity of the determinant: 
$$\det A=8\left(\prod_{j=1}^4\sin (\gamma_j)\right)
 \begin{vmatrix}
1& 1& 1 &1\\ 
\cos(\gamma_1) & \cos(\gamma_2)& \cos(\gamma_3)&\cos(\gamma_4) \\ 
\cos(2\gamma_1) & \cos(2\gamma_2)& \cos(2\gamma_3)&\cos(2\gamma_4) \\ 
\cos(3\gamma_1) & \cos(3\gamma_2)& \cos(3\gamma_3)&\cos(3\gamma_4) 
\end{vmatrix} .$$
Now, denote $x_k=\cos(\gamma_k)$ and remark that $\cos(i\gamma_k)=T_i(\cos \gamma_k)=T_i(x_k)$ where $T_i$ is the $i$th Chebyshev polynomial: 
$T_0=1, T_1=X, T_2= 2X^2-1, T_3=4X^3-3X$. 
We have $T_2+T_0=2X^2$ and $T_3+3T_1=4X^3$.
Then, doing $L_3 \leftarrow L_3+L_1$, and $L_4 \leftarrow L_4+3L_2$:
$$\det A=8\left(\prod_{j=1}^4\sin (\gamma_j)\right)
\begin{vmatrix}
1& 1& 1 &1\\ 
x_1 & x_2 & x_3 & x_4\\ 
2x_1^2 &2 x_2^2 & 2x_3^2 & 2x_4^2\\ 
4x_1^3 &4 x_2^3 & 4x_3^3 & 4x_4^3
\end{vmatrix} 
=64\left(\prod_{j=1}^4\sin (\gamma_j)\right)
\begin{vmatrix}
1& 1& 1 &1\\ 
x_1 & x_2 & x_3 & x_4\\ 
x_1^2 & x_2^2 & x_3^2 & x_4^2\\ 
x_1^3 & x_2^3 & x_3^3 & x_4^3
\end{vmatrix}$$
This is a Vandermonde matrix, hence
$$\det A
=64\left(\prod_{j=1}^4\sin (\gamma_j)\right)
\prod_{1\leq i<j\leq 4} (x_i-x_j)=64\prod_{k=1}^4\sin (\gamma_k) \prod_{1\leq i<j \leq 4} (\cos(\gamma_i)-\cos(\gamma_j)).
$$
\end{proof}

\begin{lemma}\label{CasAnnul}
Let $\gamma_1, \dots, \gamma_4$ be four reals. 
Let $\lambda_1, \dots, \lambda_4 \in \R\backslash\{0\}$ such that 
\begin{equation}
\label{eqsin}
\sum_{k=1}^4 \lambda_k \sin (l\gamma_k)=0,\qquad l=1,\dots,4.
\end{equation}
Then, one of the following cases holds:
\begin{enumerate}
\item All $\gamma_k$ are multiples of $\pi$.
\item Exactly two $\gamma_k$ are multiples of $\pi$:  $\gamma_{i_1}=\gamma_{i_2}=0 \pmod{\pi}$ 
and $\gamma_{i_3}=\pm\gamma_{i_4}\pmod{2\pi}$.
\item Only one $\gamma_k$ is multiple of $\pi$:  $\gamma_{i_1}=0 \pmod{\pi}$ and 
$\gamma_{i_2}=\pm \gamma_{i_3}=\pm\gamma_{i_4} \pmod{2\pi}$. 
\item No $\gamma_k$ is multiple of $\pi$ and 
$\gamma_1=\pm\gamma_{2}=\pm \gamma_{3}=\pm\gamma_{4} \pmod{2\pi}$.
\end{enumerate}
\end{lemma}

\begin{proof} First observe that,
since $\sum_{k=1}^4 \lambda_k \sin (l\gamma_k)=0$ with $\lambda\neq 0_{\R^4}$, necessarily
 det($A$)=0 where $A=(\sin(i \gamma_j))_{1\leq i,j\leq 4}$. Using Lemma~\ref{detA}
 \begin{equation}\label{eqdetA}
 \prod_{k=1}^4\sin (\gamma_k)  \prod_{1\leq i<j \leq 4} (\cos(\gamma_i)-\cos(\gamma_j))=0.
 \end{equation}
Now, let us study the various cases that make this quantity vanish. \\

For the first case, note that if three $\gamma_k$ are multiples of $\pi$:  $\gamma_{i_1}=\gamma_{i_2}=\gamma_{i_3}=0 \pmod{\pi}$ then 
equation  \eqref{eqsin} becomes
$ \lambda_{i_4}\sin (l\gamma_{i_4})=0$ and the last angle is also null modulo $\pi$.\\

In case 2., equation \eqref{eqsin} entails
$$ \lambda_{i_3}\sin (l\gamma_{i_3})+ \lambda_{i_4}\sin (l\gamma_{i_4})=0,\qquad l=1,2$$
with $\gamma_{i_3}\neq 0 \pmod{\pi}, \gamma_{i_4}\neq 0 \pmod{\pi}.$
Then, since $( \lambda_{i_3}, \lambda_{i_4})\neq(0,0)$,
$$0= \begin{vmatrix}
\sin(\gamma_{i_3}) & \sin(\gamma_{i_4}) \\
\sin(2\gamma_{i_3})\ & \sin(2\gamma_{i_4})
\end{vmatrix} 
=2\sin(\gamma_{i_3})\sin(\gamma_{i_4})(\cos(\gamma_{i_4})-\cos(\gamma_{i_3})).$$
Then $\cos(\gamma_{i_3})=\cos(\gamma_{i_4})$.
Either $\gamma_{i_3}=\gamma_{i_4} \pmod{2\pi}$,
or $\gamma_{i_3}=-\gamma_{i_4} \pmod{2\pi}$.\\

Let us now study case 3. For the sake of simplicity we assume that 
$\gamma_{4}=0 \pmod{\pi}$ and $\gamma_{k}\neq 0 \pmod{\pi}$ for $k=1,2,3$. 
Equation \eqref{eqsin} gives
$$ \lambda_{1}\sin (l\gamma_{1})+ \lambda_{2}\sin (l\gamma_{2})+\lambda_3\sin(l\gamma_3)=0,\qquad l=1,2,3.$$
With the same proof as  Lemma~\ref{detA}, we obtain
$$\prod_{k=1}^3\sin (\gamma_k)  \prod_{1\leq i<j \leq 3} (\cos(\gamma_i)-\cos(\gamma_j))=0.$$
Then $\gamma_1=\pm \gamma_2 \pmod{2\pi}$ or $\gamma_1=\pm \gamma_3 \pmod{2\pi}$ or $\gamma_2=\pm \gamma_3 \pmod{2\pi}$.  Moreover, if, for example, $\gamma_1=\pm \gamma_2 \pmod{2\pi}$ then 
$$ (\lambda_{1}\pm \lambda_2)\sin (l\gamma_{1})+\lambda_3\sin(l\gamma_3)=0,\qquad l=1,2$$
We are reduced to the previous case, then  $\gamma_1=\pm\gamma_3 \pmod{2\pi}$.\\

In the case 4., equation \eqref{eqdetA} becomes $ \prod_{1\leq i<j \leq 4} (\cos(\gamma_i)-\cos(\gamma_j))=0$, which provides 6 possible equalities. 
Assume, for example, $\cos(\gamma_1)-\cos(\gamma_2)=0$ and consequently $\gamma_1=\pm \gamma_2 \pmod{2\pi}$. Then 
$$( \lambda_{1}\pm \lambda_2)\sin (l\gamma_{1})+\lambda_3\sin(l\gamma_3)+ \lambda_{4}\sin (l\gamma_{4})=0,\qquad l=1,2,3.$$
Reasoning as in previous case, $\gamma_{1}=\pm \gamma_{3}=\pm\gamma_{4} \pmod{2\pi}$.

\end{proof}

\subsection{Proof of Theorem~\ref{consistance} (consistency)}
\label{proof:consistence}

This proof and the following are inspired from \cite{butucea2014}.
Let us denote $\tilde \Theta =(0,1/2) \times \S^1\times \S^1$. Denote $\dot  \phi(\theta)$ the gradient of any function $\phi$ with respect to $\theta=(p,\alpha,\beta)$, and $\ddot  \phi(\theta)$ the Hessian matrix. 

The proof of Theorem~\ref{consistance} relies on some preliminary results, given in the sequel. 

\begin{proposition}\label{contrast} Under Assumption~\ref{hyptheta2} 
the contrast function $S$ verifies  the following properties: 
$S(\theta) \geq 0$, and $S(\theta)=0$ if and only if $\theta=\theta_0$ or $\theta=\theta_0+\pi$.
\end{proposition}
\begin{proof}
It is clear that $S(\theta)\geq 0$ and that 
$$S(\theta_0)=\sum_{l =-4}^{4} \left ( \Im \left ( g^{\star l } \overline{{M}^l(\theta_0)} \right ) \right )^2=\sum_{l =-4}^{4} \left ( \Im \left ( f^{\star l } |M^l(\theta_0)|^2\right ) \right )^2=0.$$
By Lemma~\ref{EqMtheta}, if $\theta \neq \theta_0 \pmod{\pi}$, 
there exists $l_1\in \{1,\dots, 4\}$ such that

$ \Im\left(M^{l_1}(\theta_0)\overline{M^{l_1}(\theta)}\right) \neq 0$ %
so that 
$
S(\theta)\geq  \left ( \Im \left ({ g^{\star l_1}}\overline{{M}^{l_1}(\theta)} \right ) \right )^2 >0.
$
\end{proof}

\begin{lemma}\label{bornes}  
\begin{enumerate}
\item For all $\theta$ in $\tilde\Theta$, $|M^l(\theta)| \leq 1$. 
\item For all $1\leq k\leq n$, for all $l$ in $\Z$, \[
\sup_{\theta\in \tilde\Theta} |  Z^l_{k}(\theta) |\leq \frac{1}{2\pi}, \qquad \sup_{\theta\in \tilde\Theta} |  J^l(\theta)|  \leq \frac{1}{2\pi}.
\]
\item For all $1\leq k\leq n$, for all $l$ in $\Z$, 
{ 
\[
\sup_{\theta\in \tilde\Theta} \| \dot  Z^l_{k}(\theta) \|\leq \frac{ 2 +| l| }{ \sqrt{2}\pi}, \qquad \sup_{\theta\in \tilde\Theta} \| \dot  J^l(\theta)\|  \leq \frac{ 2 +| l| }{ \sqrt{2}\pi} .
\]
}
where $\|.\|$ is the Euclidean norm.
\item For all $1\leq k\leq n$, for all $l$ in $\Z$, \[
\sup_{\theta\in \tilde\Theta} \|  \ddot Z^l_{k}(\theta) \|_F\leq  \frac{|l| + l^2}{\pi} , \qquad \sup_{\theta\in \tilde\Theta} \| \ddot  J^l(\theta)\|_F  \leq  \frac{|l| + l^2}{\pi}.
\]
where $\|.\|_F$ is the Frobenius norm.
\end{enumerate}

\end{lemma}

\begin{proof}
Point 1 is straightforward.

2. Let us start with $ Z^l_{ k}(\theta)$. We recall that
$
Z_{ k}^l(\theta)= \Im \left (  \frac{ e^{i lX_k}}{2 \pi}  M^l(\theta) \right ). $
Then
$$
 |Z_{k}^l(\theta)| \leq  \frac{1}{2 \pi} |M^l(\theta)|  \leq \frac{ 1}{2 \pi}.
$$
Furthermore
\[
|  J^l(\theta) |  \leq  {|g^{\star l}| }{|M^l(\theta)|}   \leq \frac{ 1}{2\pi} \int_{\S_1} g \leq  \frac{1 }{2 \pi}.
\] 

3. We have 
\[
\dot Z_{k}^l(\theta)  =
\frac{1}{2\pi}\Im \left( e^{i lX_k} \dot M^{l }(\theta)\right)= \frac{1}{2\pi}\Im \left ( e^{i lX_k}
\begin{pmatrix} e^{-il\alpha} - e^{-il  \beta}\\
                         -il pe^{-i \alpha l}\\
                          - i l(1-p) e^{-i \beta l}\\
\end{pmatrix}\right)
\]
and 
\[ 
\dot J^l(\theta) =    \Im \left ( \overline{g^{\star l}} \dot M^{l }(\theta) \right)= \Im \left (  \overline{g^{\star l}}\begin{pmatrix} e^{-il\alpha} - e^{-il \beta}\\
                         -il pe^{-i \alpha l}\\
                          - i l(1-p) e^{-i \beta l}\\
\end{pmatrix}\right).
\]
 We get 
 { 
 $$\|\dot Z_{k}^l(\theta)  \|  \leq    \frac{1}{2 \pi} \left(2^2 + p^2l^2 + (1-p)^2 l^2\right)^{1/2} \leq \frac{ 2 +| l| }{ \sqrt{2}\pi}$$
 }
and   we have the same bound for $\|\dot J^l(\theta)  \| $.

4. We have 
\begin{eqnarray*}
\ddot Z_{ k}^l(\theta) &=& \Im \left ( \frac{ e^{i lX_k}}{ 2\pi }    \ddot M^{l }_{}(\theta) \right ) \\
&=&  \Im \left ( \frac{ e^{i lX_k}}{ 2\pi}   \begin{pmatrix} 0 & -ile^{-il\alpha} & il e^{-il\beta}\\
                         - il e^{-il\alpha } & -l^2 pe^{-il \alpha} & 0\\
                          ile^{-il \beta} & 0 & -l^2(1-p)  e^{-i \beta l}
                      \end{pmatrix}  \right ).
                          \end{eqnarray*}
Thus
$$
\| \ddot Z^l_{k}(\theta) \|_F \leq \frac{1}{2\pi}\left(4l^2+l^4p^2+l^4(1-p)^2\right)^{1/2}\leq  \frac{|l| + l^2}{\pi}.
$$
We bound $\| \ddot J^l (\theta) \|_F$ in the same way. This ends the proof of the lemma.
\end{proof}

\begin{lemma}\label{bornes2} There exists a numerical positive constant $C$ such that the following inequalities hold.

1. For all $1\leq k\leq n$, for all $l$ in $\Z$
\[ \forall \theta,\theta'\in \tilde\Theta \qquad 
\| \dot Z^l_{k}(\theta) -\dot Z^l_{k}(\theta') \| \leq C \| \theta- \theta'\| (1+|l|+l^2).
\]
2. We also have 
\[
\|  \ddot Z^l_{k}(\theta) -\ddot Z^l_{k}(\theta')   \|_F  \leq C \| \theta - \theta' \| (1 +| l| +l^2+| l|^3).
\]
\end{lemma}

\begin{proof}
We use Taylor expansions at first order and then apply same bounding techniques as in Lemma \ref{bornes}. 
\end{proof}

\begin{lemma} \label{lip} 
\begin{enumerate}
\item The function $S$ is Lipschitz continuous over $\tilde \Theta$.
\item The function  $S_n(\theta)$ is Lipschitz continuous over $\tilde \Theta$.
\item {The function  $\ddot S_n(\theta)$ is Lipschitz continuous over $\tilde \Theta$ with respect to Frobenius norm}, 
 {with Lipschitz constant not depending on $n$}.
\end{enumerate}
\end{lemma}

\begin{proof}
We will write $C$ for a numerical constant that may change from line to line but is numerical.

Let us start with point 1. We recall that $S(\theta)= \sum_l J^{l}(\theta)^2$. Let $\theta$ and $\theta'$ in $\tilde \Theta$. As $\tilde \Theta$ is a convex set, we get, thanks to the mean value theorem
\begin{eqnarray*}
|S(\theta)-S(\theta')| &= &\left| \sum_{l=-4}^4 J^{l}(\theta)^2-J^{l}(\theta')^2\right| 
=\left |2(\theta -\theta')^\top \sum_{l=-4}^4 J^l(\theta_u)\dot J^{l}(\theta_u)  \right|\\
&\leq & C \| \theta -\theta' \|\sum_{l=-4}^4 (1+ |l|) 
\leq C\| \theta -\theta' \|
\end{eqnarray*}
with $\theta_u$ lying on the line connecting  $\theta$ to $\theta'$, and using Lemma \ref{bornes}.

Let us shift to point 2. Due to the mean value theorem, we have 
\begin{eqnarray*}
|S_n(\theta) - S_n(\theta') |&=& \left| \frac{1}{n(n-1)} \sum_{k \neq j} \sum_{l=-4}^4 \left ( Z^l_{k}(\theta) Z^l_{j}(\theta) - Z^l_{k}(\theta') Z^l_{j}(\theta')  \right ) \right |\\
&=& \left| \frac{1}{n(n-1)}  \sum_{k \neq j} \sum_{l=-4}^4 \left ( (\theta- \theta')^\top \nabla  [ Z^l_{k}(\theta) Z^l_{j}(\theta)] |_{\theta= \theta_u}  \right ) \right |\\
&=&\left| \frac{2(\theta -\theta')^\top}{n(n-1)}  \sum_{k \neq j} \sum_{l=-4}^4  \dot Z^{l}_{k}(\theta_u)Z^{l}_{j}(\theta_u) \right |,
\end{eqnarray*}
with $\theta_u$ lying on the line connecting  $\theta$ to $\theta'$. Then using 1. and 2. of Lemma  \ref{bornes} we get
\begin{eqnarray*}
 | S_n(\theta) - S_n(\theta') | &\leq&  \frac{C \| \theta- \theta' \|}{n(n-1)}  \sum_{k \neq j} \sum_{l=-4}^4   (1+|l |) \leq   C \| \theta- \theta' \|
\end{eqnarray*}
which ends the proof of the second point.

Concerning point 3. we have that
$$
\ddot S_n(\theta)=  \frac{2}{n(n-1)}  \sum_{k \neq j} \sum_{l=4}^4 (\ddot Z^l_{k}(\theta) Z^l_{j}(\theta)+ \dot Z^l_{k}(\theta) \dot Z^l_{j}(\theta)^\top). 
$$
Hence 
\begin{eqnarray*}
\| \ddot S_n(\theta) - \ddot S_n(\theta')  \|_F  
&\leq &\frac{2}{n(n-1)} \sum_{k \neq j} \sum_{l=-4}^4 \bigg (  \|  (\ddot Z^l_{k}(\theta)  - \ddot Z^l_{k}(\theta')) Z^l_{j}(\theta)   \|_F \\
&&+  \|  \ddot Z^l_{k}(\theta' )(  Z^l_{j}(\theta)- Z^l_{j}(\theta') ) \|_F 
+ \| \dot Z^l_{k}(\theta' )(  \dot Z^l_{j}(\theta)- \dot Z^l_{j}(\theta')^\top ) \|_F \\
&  &  +
\|  ( \dot  Z^l_{k}(\theta')- \dot Z^l_{k}(\theta) ) \dot Z^l_{j}(\theta)^\top  \|_F \bigg )
\end{eqnarray*}
Using Taylor expansions and Lemma \ref{bornes} and \ref{bornes2}, we get that 
$$
\| \ddot S_n(\theta) - \ddot S_n(\theta')  \|_F \leq C\| \theta -\theta' \| \sum_{l=-4}^4 (1+ |l|+ l^2 +|l|^3) .
$$

\end{proof}

\begin{proposition}\label{variance} There exist a positive  constant $C$ 
 such that 
\[
 \sup_{\theta\in \tilde \Theta} \E [(S_n(\theta)- S(\theta))^2] \leq  \frac{C}{n}.  
\]
\end{proposition}

\begin{proof} The definitions of $S_n$ and $S$ provide
\[
S_n(\theta)- S(\theta)  =\frac{1}{n(n-1)}   \sum_{l=-4}^4\sum_{k \neq  j} \left( Z^l_{k}(\theta) { Z^l_{ j}(\theta) } -  J^{l}(\theta)^2 \right)
=T_n + V_n
\]
where  
\[ 
T_n=  \frac{2}{n(n-1)}  \sum_{l=-4}^4  \sum_{k < j}( Z^l_{k}(\theta) - J^l(\theta)) ( {Z^l_{ j}(\theta) }- {J^l(\theta)}) 
\]
and
\[V_n =
   \frac{2}{n}   \sum_{l=-4}^4 \sum_{k=1 }^n  (Z^l_{k}(\theta)- {J^l(\theta)})  J^l(\theta).
\]
Note that $\E(Z^l_{k}(\theta) - J^l(\theta))=0$ which entails $\E[T_nV_n]=0$.
Then 
\[
\E\left[\left(S_n(\theta)- S(\theta) \right)^2\right]=\E\left[\left(T_n+V_n\right)^2\right]=\E\left[T_n^2\right]+\E\left[V_n^2\right].
\]
 Now, 
since the variables $ \left(\sum_{l=-4}^4  ( Z_{k}^l(\theta)  -  J^l(\theta)) ( {Z_{j}^l(\theta) } -  J^l(\theta))\right)_{k<j}$  are uncorrelated,
\begin{eqnarray*}
\E[T_n^2] &=& \frac{2}{n(n-1)}  \E \left [ \left  (  \sum_{l=-4}^4  ( Z_{1}^l(\theta)  -  J^l(\theta)) ( {Z_{2}^l(\theta) } - { J^l(\theta))}  \right ) ^2 \right ] \\
& \leq &  \frac{2}{n(n-1)} \E \left [   \left (   \sum_{l=-4}^4  \frac{2}{2\pi} \cdot \frac{2}{2\pi} \right )^2  \right ]   
\leq  \frac{ C}{2n} 
\end{eqnarray*}
using Lemma \ref{bornes}.
We focus now on $V_n$: in the same way 
\begin{eqnarray*}
\E[V_n^2] &=& \frac 4 n \E \left [ \left  ( \sum_{l=-4}^4  (Z_{1}^l(\theta)-J^l(\theta)) {J^l(\theta)  } \right )^2 \right ] \\
& \leq & \frac 4 n \E \left [ \left  ( \sum_{l=-4}^4   \frac{2}{2\pi} \cdot \frac{1}{2\pi}    \right )^2 \right ] \leq  \frac{ C}{2n}  ,
\end{eqnarray*}
using Lemma \ref{bornes} again. 
\color{black}
\end{proof}

{ Theorem~\ref{consistance} is finally  proved using the following lemma, its assumptions 
being  ensured by Proposition~\ref{contrast}, Lemma~\ref{lip} and Proposition~\ref{variance}.
}

\begin{lemma}\label{lem:convmin}
 Assume that $\Theta$ is a compact set and let $S: \Theta \to \R$ be a continuous function. Assume that 
$$S(\theta)=\min_{\Theta} S \Leftrightarrow \theta=\theta_0 \text{ or }\theta=\theta_0'$$ where $\theta_0,\theta_0'\in \Theta$. Let $S_n:\Theta \to \R$ be a  function which is uniformly continuous and such that for all $\theta$
$| S_n(\theta)-S(\theta)|$ tends to 0 in probability.
Let $\hat\theta_n$ be a point such that $S_n(\hat\theta_n)=\inf_{\Theta} S_n$. Then $\hat \theta_n \rightarrow  \theta_0$ or $\theta_0'$  in probability. 
\end{lemma}

This is a classical result in the theory of minimum contrast estimators, when $\theta_0=\theta'_0$ (see \cite{vandervaart} or \cite{DCD}).
We reproduce the proof  since it is slightly adapted to the case of two argmins. 

\begin{proof} Let $\epsilon>0$ and $B$ be the union of the open ball with center $\theta_0$ and radius $\epsilon$
and the open ball with center $\theta_0'$ and radius $\epsilon$. 
Since $S$ is continuous and $B^c\subset\Theta$ is a compact set, there exists $\theta_\epsilon \in B^c$ such that 
$S(\theta_\epsilon)=\inf_{B^c} S$.
Using the assumption, since $\theta_\epsilon\neq \theta_0$, and $\theta_\epsilon\neq \theta_0'$
$$\delta:=S(\theta_\epsilon)-S(\theta_0)>0.$$
Since $S_n$ is uniformly continuous, there exists $\alpha>0$ such that 
$$\forall \theta, \theta' \qquad \|\theta-\theta'\|< \alpha \Rightarrow
|S_n(\theta)-S_n(\theta')|\leq \delta/2.$$
Moreover $B^c$ is a compact set then there exists a finite set $(\theta_i)$ such that $B^c\subset \cup_{i=1}^I B(\theta_i,\alpha)$.
Denote $\Delta_n:=\max_{0\leq i \leq I}| S_n(\theta_i)-S(\theta_i)|$. The assumption ensures that $\Delta_n$ tends to 0 in probability.
Let $\theta\in B^c$. There exists $1\leq i\leq I$ such that $\|\theta-\theta_i\|<\alpha$, and then $|S_n(\theta)-S_n(\theta_i)|\leq \delta/2$.
Thus 
\begin{align*}
S_n(\theta)-S_n(\theta_0)&=(S_n(\theta)-S_n(\theta_i))+(S_n(\theta_i)-S(\theta_i))\\
&+(S(\theta_i)-S(\theta_0))+(S(\theta_0)-S_n(\theta_0))\\
& \geq -\delta/2-\Delta_n+\delta-\Delta_n
\end{align*}
using that $S(\theta_i)-S(\theta_0)\geq S(\theta_\epsilon)-S(\theta_0)=\delta$.
Then \begin{align*}
\inf_{\theta\in B^c} S_n(\theta)-S_n(\theta_0) \geq \delta/2-2\Delta_n.
\end{align*}
Now, if $\|\hat \theta_n-\theta_0\|\geq \epsilon$ and $\|\hat \theta_n-\theta_0'\|\geq \epsilon$
then $\hat\theta_n \in B^c$ and 
$$\inf_{\theta \in \Theta} S_n(\theta)=S_n(\hat\theta_n)=\inf_{\theta \in B^c} S_n(\theta).$$
In particular $\inf_{\theta\in B^c} S_n(\theta)\leq S_n(\theta_0)$ so that 
\begin{eqnarray*}
\P(\|\hat \theta_n-\theta_0\|\geq \epsilon \text{ and }\|\hat \theta_n-\theta_0'\|\geq \epsilon)&\leq &\P(0\geq \inf_{\theta \in  B^c} S_n(\theta)- S_n(\theta_0) \geq \delta/2- 2\Delta_n)\\
&\leq &
\P(\Delta_n\geq \delta/4) \longrightarrow 0
\end{eqnarray*}
since $\Delta_n$ tends to 0 in probability.
\end{proof}
\color{black}

\subsection{Proof of Theorem~\ref{normasymp} (asymptotic normality)}
\label{proof:normasymp}

The Taylor's theorem and the definition of $\hat\theta_n$ give
$$
\dot S_n(\hat \theta_n) =  \dot{S_n}(\theta_0) + \ddot{S_n}(\theta_n^*)(\hat \theta_n -\theta_0)=0,
$$
where $\theta_n^*$ lies in the line segment with extremities $\theta_0$ and $\hat \theta_n$. Equivalently we have,
$$
 \ddot{S_n}(\theta_n^*)(\hat \theta_n -\theta_0) = - \dot{S_n}(\theta_0).
$$
We recall that 
$$
S_n(\theta_0)= \frac{1}{n(n-1)} \sum_{k \neq j} \sum_{l=-4}^4 Z^l_{k}(\theta_0) Z^l_{j}(\theta_0) 
$$
and 
$$
\dot S_n(\theta_0)= \frac{2}{n(n-1)} \sum_{k \neq j} \sum_{l=-4}^4 \dot Z^l_{k}(\theta_0) Z^l_{j}(\theta_0) 
$$
and
$$
\ddot S_n(\theta_0)=  \frac{2}{n(n-1)}  \sum_{k \neq j} \sum_{l=-4}^4\ddot Z^l_{k}(\theta_0) Z^l_{j}(\theta_0)+ \dot Z^l_{k}(\theta_0) \dot Z^l_{j}(\theta_0)^\top.
$$
\textbf{Step 1}- Let us prove that 
$$
\sqrt{n} \dot S_n(\theta_0) \overset{d}{\longrightarrow} \mathcal{N}(0,V). 
$$
We remind by Lemma~\ref{EqMtheta} that $
J^l(\theta_0)=0.
$
Hence
$$
\E(\dot S_n(\theta_0))= 2 \sum_{l=-4}^4 \dot J^l(\theta_0) J^l(\theta_0) =0.
$$
We can break down $\dot S_n(\theta_0)$ in the following way:
\begin{eqnarray*}
\dot S_n(\theta_0) &=& \frac{2}{n(n-1)}   \sum_{k \neq j} \sum_{l=-4}^4 (\dot Z^l_{k}(\theta_0) - \dot J^l(\theta_0) +  \dot J^l(\theta_0)) Z^l_{j}(\theta_0) \\
&=&  \frac{4}{n(n-1)} \sum_{k< j} \sum_{l=-4}^4( \dot Z^l_{k}(\theta_0) -  \dot J^l(\theta_0)) Z^l_{j}(\theta_0)  
+ \frac{2}{n} \sum_{j=1}^n \sum_{l=-4}^4  \dot J^l(\theta_0)  Z^l_{j}(\theta_0)  \\
&=:& A_n + B_n.
\end{eqnarray*}
Note that $A_n$ and $B_n$ are centered variables.  Let us show  that $\sqrt{n}A_n= o_P(1)$.
Note that the variables $W_{jk}:= \left(\sum_{l=-4}^4  ( \dot Z_{k}^l(\theta_0)  -  \dot J^l(\theta_0))  Z_{j}^l(\theta_0) \right)_{k<j}$  are centered and uncorrelated. Then
$$\E(\|A_n\|^2)=\E\left(\left\|\frac{4}{n(n-1)} \sum_{k< j}W_{jk}\right\|^2\right)
=\frac{8}{n(n-1)} \E\|W_{12}\|^2.$$
Using Lemma \ref{bornes}, there exists $C>0$ such that 
$$\|W_{12}\|\leq \sum_{l=-4}^4 \frac{2(1+|l|)}{\sqrt{2}\pi}\frac{1}{2\pi}\leq C$$ so that $\E(\|\sqrt{n}A_n\|^2)\leq {8C^2}/{(n-1)} $.
Finally, invoking Markov inequality we have that $\sqrt{n}A_n= o_P(1)$. 
 We can write $\sqrt{n}B_n$ in the following way: 
 \[
 \sqrt{n} B_n =  \frac{2}{\sqrt{n}} \sum_{k=1}^n U_k(\theta_0) ,
 \]
where  we set  $U_{k}(\theta_0):= \sum_{l =-4}^4   \dot J^l(\theta_0) Z^l_{k}(\theta_0) $.
  Note that  the $U_{k}(\theta_0)$'s   are i.i.d and centered.  Invoking the central limit theorem, we have that
  $$
  \frac{1}{\sqrt{n}} \sum_{k=1}^n U_k(\theta_0) \overset{d}{\longrightarrow} \mathcal{N}(0,V/4),
  $$
where $V/4$ is the covariance matrix of $U_1(\theta_0)$, equal to $\E(U_1(\theta_0)U_1(\theta_0)^\top)$.

\medskip
\textbf{Step 2}- Let us prove that $\ddot S_n(\theta_n^*) \overset{P}{\longrightarrow}  \mathcal{A}(\theta_0)$
where $\mathcal{A}(\theta_0)= 2 \sum_{l=-4}^4 \dot J^l(\theta_0) \dot J^l(\theta_0)^\top $.  
First, we have
\begin{eqnarray*}
\E(\ddot S_n(\theta_0)) &=& \ddot S(\theta_0)=2\sum_{l=-4}^4 (\ddot J^l(\theta_0) \underbrace{J^l(\theta_0)}_{=0}+ \dot J^l(\theta_0) \dot J^l(\theta_0)^\top) \\
&=& 2 \sum_{l=-4}^4 \dot J^l(\theta_0) \dot J^l(\theta_0)^\top=\mathcal{A}(\theta_0)  .
\end{eqnarray*}
Next we write the decomposition
$$
\| \ddot S_n(\theta_n^*)  - \mathcal{A}(\theta_0) \|_F \leq \| \ddot S_n(\theta_n^*) -  \ddot S_n(\theta_0)   \|_F  + \| \ddot S_n(\theta_0) - \E \ddot S_n(\theta_0)  \|_F
$$
We get due to the Lipschitz property of $\ddot S_n$ stated in Lemma \ref{lip} that
\[
\P  \left ( \| \ddot S_n(\theta_0)  -\ddot S_n(\theta_n^*) \|_F  \geq \varepsilon \right ) \leq \P  \left ( K\| \theta_n^*  -\theta_0 \|  \geq \varepsilon \right ) \rightarrow 0,
\]
 because $\hat \theta_n \overset{P}{\longrightarrow} \theta_0.$  
Last, let us focus on the term $ \| \ddot S_n(\theta_0) - \E \ddot S_n(\theta_0)  \|_F$. We  remind that  
$$
\ddot S_n(\theta_0)- \E \ddot S_n(\theta_0) =  \frac{2}{n(n-1)}  \sum_{k \neq  j} \sum_{l=-4}^4 \left(\ddot Z^l_{k}(\theta_0) Z^l_{j}(\theta_0)+ \dot Z^l_{k}(\theta_0) \dot Z^l_{j}(\theta_0)^\top
-\dot J^l(\theta_0) \dot J^l(\theta_0)^\top\right).
$$
From now on, we drop indices $l$ and $\theta_0$ to simplify the notation. We center the variables in order to find uncorrelatedness:
\begin{eqnarray*}
\ddot Z_{k} Z_{j}+\dot Z_{k} \dot Z_{j}^\top-\dot J\dot J^\top=
\underbrace{\left(\ddot Z_{k}-\ddot J\right) Z_{j}}_{A}+
\underbrace{\ddot J Z_j}_{B}+ 
\underbrace{(\dot Z_{k} -\dot J)(\dot Z_{j}-\dot J)^\top}_{C}
\\+\underbrace{\dot J(\dot Z_{j}-\dot J)^\top}_{D}
+\underbrace{( \dot Z_{k} -\dot J)(\dot J)^\top}_{E}
\end{eqnarray*}
(remind that $\E(Z_j)=J^l(\theta_0)=0$). Then 
$
\ddot S_n(\theta_0)- \E \ddot S_n(\theta_0) = 2\sum_{l=-4}^4 (A+B +C+D+E)
$
where 
\begin{align*}
A&=\frac{2}{n(n-1)}  \sum_{k < j}\left(\ddot Z_{k}-\ddot J\right) Z_{j}\\
B&=\frac{1}{n}  \sum_{ j=1}^n\ddot JZ_j\\
C&=\frac{2}{n(n-1)}  \sum_{k < j}(\dot Z_{k} -\dot J)(\dot Z_{j}-\dot J)^\top\\
D&=\frac{1}{n}  \sum_{j=1}^n\dot J(\dot Z_{j}-\dot J)^\top\\
E&=\frac{1}{n}  \sum_{k =1}^n ( \dot Z_{k} -\dot J)\dot J^\top=D^\top
\end{align*}
Using the weak law of large numbers for uncorrelated centered variables, we obtain that 
$\| \ddot S_n(\theta_0)  - \E \ddot S_n(\theta_0)  \|_F \overset{P}{\to} 0$ which completes the step 2.

Finally it is sufficient to apply Slutsky's Lemma to obtain
the theorem.

\subsection{Estimation of the covariance}\label{sec:estimV}

\begin{proposition}\label{estimV} Consider notation and assumptions of Theorem~\ref{normasymp}.
 Let $V=4\E(U_1U_1^\top)$ where $U_1= \sum_{l=-4}^4 Z^l_{1}(\theta_0)\dot J^l(\theta_0) .$ Then $$\frac{4}{n^3}\sum_{1\leq k, j,j'\leq n}\sum_{-4\leq l,l'\leq 4} Z_k^{l}(\hat\theta_n)Z_k^{l'}(\hat\theta_n)\dot Z_j^{l}(\hat\theta_n)(\dot Z_{j'}^{l'}(\hat\theta_n))^\top.$$ tends almost surely toward $V$ when $n$ tends to $+\infty$.
\end{proposition}

Thus we obtain a consistent estimator for $V$ (that allows to estimate the covariance $\Sigma$). Nevertheless this estimator is biased. Notice that the quantity
$$ \frac{4}{n(n-1)(n-2)}\sum_{k=1}^n\sum_{j\neq k}\sum_{j'\notin\{k,j\}}\sum_{-4\leq l,l'\leq 4} Z_k^{l}(\theta_0)Z_k^{l'}(\theta_0)\dot Z_j^{l}(\theta_0)(\dot Z_{j'}^{l'}(\theta_0))^\top$$
has expectation 
$$ 4\sum_{-4\leq l,l'\leq 4} \E[Z_1^{l}(\theta_0)Z_1^{l'}(\theta_0)]\dot J^{l}(\theta_0)(\dot J^{l'}(\theta_0))^\top=V$$
and we could also prove (with some additional technicalities in the following proof about the uniform convergence in $k$) that it tends almost surely toward $V$.
However, we lose the "unbiased" property when replacing $\theta_0$ by $\hat \theta_n$.

\subsubsection*{Proof of Proposition~\ref{estimV}}
Let $U_k=\sum_{l=-4}^4 Z_k^l(\theta_0)\dot J^l(\theta_0)$. The law of large numbers gives
$$V=\E(4U_1U_1^\top)=\lim_{n\to \infty} \frac{4}{n}\sum_{k=1}^n U_kU_k^\top$$
where the convergence is almost sure.
Moreover
$$U_kU_k^\top=\sum_{-4\leq l,l'\leq 4} Z_k^{l}Z_k^{l'}\dot J^{l}(\dot J^{l'})^\top=\lim_{n\to \infty}\frac1{n^2}
\sum_{-4\leq l,l'\leq 4} Z_k^{l}Z_k^{l'}\sum_{1\leq j,j'\leq n}\dot Z_j^{l}(\dot Z_{j'}^{l'})^\top$$
where the convergence is almost sure and we have dropped the $\theta_0$ for the sake of simplicity. This convergence is uniform in $k$ in the following sense: there exists a set with probability 1 for which for any $\varepsilon>0$, there exists $N\geq 1$ such that for all $n\geq N$ and for all $1\leq k\leq n$
$$\left\| \frac1{n^2}
\sum_{-4\leq l,l'\leq 4} Z_k^{l}Z_k^{l'}\sum_{1\leq j,j'\leq n}\dot Z_j^{l}(\dot Z_{j'}^{l'})^\top-U_kU_k^\top\right\|\leq \varepsilon$$
Indeed 
\begin{align*}
&\left\| \frac1{n^2}
\sum_{-4\leq l,l'\leq 4} Z_k^{l}Z_k^{l'}\sum_{1\leq j,j'\leq n}\dot Z_j^{l}(\dot Z_{j'}^{l'})^\top-U_kU_k^\top\right\|\\
&=
\left\| 
\sum_{-4\leq l,l'\leq 4} Z_k^{l}Z_k^{l'}\left(\frac1{n^2}\sum_{1\leq j,j'\leq n}\dot Z_j^{l}(\dot Z_{j'}^{l'})^\top-\dot J^{l}(\dot J^{l'})^\top\right)\right\|\\
&\leq  \frac{1}{4\pi^2}\sum_{-4\leq l,l'\leq 4} \left\|\left(\frac1{n^2}\sum_{1\leq j,j'\leq n}\dot Z_j^{l}(\dot Z_{j'}^{l'})^\top-\dot J^{l}(\dot J^{l'})^\top\right)\right\|.
\end{align*}
Then we use the following lemma:
"If $v_{nk} \to v_k$ uniformly,  with $(v_{nk})$ and $(v_k)$ bounded, and if $n^{-1}\sum_{k=1}^n v_k \to v$ then $n^{-1}\sum_{k=1}^n v_{nk} \to v$."
To prove this lemma, notice that, for a given positive $\varepsilon$, for  $n$ large enough
\begin{eqnarray*}
\left|\frac{1}{n}\sum_{k=1}^n v_{nk}-v\right|
&\leq & \frac{1}{n}\sum_{k=1}^{N} \left|v_{nk}-v_k\right|+
\frac{1}{n}\sum_{k=N+1}^n \left| v_{nk}-v_k\right|+\left|\frac{1}{n}\sum_{k=1}^n v_{k}-v\right| \\
&\leq & \frac{N}{n}(\sup_{kn}|v_{nk}|+\sup_k |v_k|)+\frac{n-N}{n}\varepsilon+ \varepsilon \leq 3\varepsilon.
\end{eqnarray*}
That provides
$$V=\lim_{n\to \infty} \frac{4}{n^3}\sum_{1\leq k, j,j'\leq n}\sum_{-4\leq l,l'\leq 4} Z_k^{l}Z_k^{l'}\dot Z_j^{l}(\dot Z_{j'}^{l'})^\top$$
where the convergence is almost sure. Here all the $Z_k$ are depending on $\theta_0$, but we can use the consistency of $\hat\theta_n$ to finally assert
$$V=\lim_{n\to \infty} \frac{4}{n^3}\sum_{1\leq k, j,j'\leq n}\sum_{-4\leq l,l'\leq 4} Z_k^{l}(\hat\theta_n)Z_k^{l'}(\hat\theta_n)\dot Z_j^{l}(\hat\theta_n)(\dot Z_{j'}^{l'}(\hat\theta_n))^\top.$$

{ 
\subsection{Proof of Proposition~\ref{varthetachap}}
\label{sec:proof1varthetachap}}

 
We use the proof of Theorem~\ref{normasymp}. We  have seen that 
$$
 \ddot{S_n}(\theta_n^*)(\hat \theta_n -\theta_0) = - \dot{S_n}(\theta_0),
$$
with $\theta_n^*$ in the line segment with extremities $\theta_0$ and $\hat \theta_n$.
Recall that $\dot S_n(\theta_0) =A_n+B_n$ with 
\begin{eqnarray*}
 A_n &= & \frac{4}{n(n-1)} \sum_{k< j} \sum_{l=-4}^4( \dot Z^l_{k}(\theta_0) -  \dot J^l(\theta_0)) Z^l_{j}(\theta_0) \\
 B_n &=&
 =\frac{2}{n} \sum_{k=1}^n U_k(\theta_0) 
 \end{eqnarray*}
where   $U_{k}(\theta_0):= \sum_{l =-4}^4   \dot J^l(\theta_0) Z^l_{k}(\theta_0).$
  Note that  the $U_{k}(\theta_0)$'s   are i.i.d and centered so that 
 $$ \E\left\|\sum_k U_k(\theta_0)\right\|^2
 =\sum_{j=1}^3\Var\left(\sum_k U_{kj}(\theta_0)\right) 
 =\sum_{j=1}^3n\Var\left(U_{1j}(\theta_0)\right) 
 \leq n c_1$$
 using Lemma~\ref{bornes}. Here $c_1$ is a numerical constant. 
  Thus $\E\|B_n\|_1^2\leq  4c_1/n$.
In the same way  the variables $W_{jk}:= \left(\sum_{l=-4}^4  ( \dot Z_{k}^l(\theta_0)  -  \dot J^l(\theta_0))  Z_{j}^l(\theta_0) \right)_{k<j}$  are centered and uncorrelated, and also bounded.
Then
$$\E(\|n(n-1)A_n\|^{2})=\E\left(\left\| \sum_{k< j}W_{jk}\right\|^{2}\right)
=n(n-1)\E\left(\left\| W_{12}\right\|^{2}\right)\leq n(n-1)c_2.$$
Then $\E\left( \|A_n\|^{2}\right)\leq c_2/n$ and 
$\sup_n\E\left( n\|\dot S_n(\theta_0)\|^{2}\right)\leq 8c_1+2c_2<\infty$. 

In  the proof of Theorem~\ref{normasymp}, we noted that $\ddot S_n(\theta_n^*)
$ tends to  $\ddot S(\theta_0)$ in probability. Actually we can prove that  the convergence is almost sure. 
Indeed the strong law of large numbers is true for uncorrelated variables if their second moments have a common bound (see e.g. \cite{chung2001}) so that $$\ddot S_n(\theta_0)  -\ddot S(\theta_0)=\ddot S_n(\theta_0)  - \E \ddot S_n(\theta_0)  \overset{a.s.}{\longrightarrow} 0.$$
Since $\ddot S_n$ is continuous, it is sufficient to show that   the convergence of $\hat\theta_n$ towards $\theta$ is almost sure and this will imply that $\ddot S_n(\theta^*_n)$ converges almost surely towards $\ddot S_n(\theta_0)$ (recall that $\theta_n^*$ in the line segment with extremities $\theta_0$ and $\hat \theta_n$).
To do this, remark first that $ S_n(\theta)  - S(\theta)  \overset{a.s.}{\longrightarrow} 0$ by the strong law of large numbers for uncorrelated variables again (see the decomposition of $S_n-S$ in the proof of Proposition~\ref{variance}). Now, we come back to the proof of Lemma~\ref{lem:convmin} (in the case of a unique minimum $\theta_0$), with this new assumption that $S_n(\theta)$ tends to $S(\theta)$ almost surely. The proof shows that for any $\epsilon>0$ there exist $\delta(\epsilon)>0$ and $\Delta_n(\epsilon)$ which tends to 0 almost surely  such that
$$\|\hat\theta_n-\theta_0\|\geq \epsilon\Rightarrow \Delta_n(\epsilon)\geq \delta(\epsilon)/4.$$
Let $\Gamma=\cap_{p\geq 1}\{\Delta_n(1/p)\to 0\}$. This set has probability 1 and on this set, for any $\varepsilon>0$, taking $p \geq 1/\varepsilon$, there exists $N\geq 1$ such that for any $n\geq N$
$$\Delta_n(1/p)< \delta(1/p)/4 \text{ and then  } \|\hat\theta_n-\theta_0\|< (1/p)\leq \varepsilon.$$
This ensures that on the set $\Gamma$, $\hat\theta_n$ tends to $\theta_0$, and finally $\ddot S_n(\theta_n^*)$ tends to $\ddot S(\theta_0)$  almost surely.

Now, since $\ddot S(\theta_0)$ is assumed invertible, there exists $n_1$ such that for all $n\geq n_1$, $\ddot S_n(\theta_n^*) $ is invertible and 
$\|\ddot S_n(\theta_n^*)^{-1}\|_{op}\leq 2\|\ddot S_n(\theta_0)^{-1}\|_{op}:=C(\theta_0)$ a.s. 
Then 
$$n\|\hat \theta_n-\theta_0\|^{2}\leq
C(\theta_0)^2 n\|\dot S_n(\theta_0)\|^{2}\quad \text{ a.s.}$$
and 
$$\E(n\|\hat \theta_n-\theta_0\|^{2})\leq
C(\theta_0)^2\E( n\|\dot S_n(\theta_0)\|^{2})\leq  C(\theta_0)^2(8c_1+2c_2).$$
Moreover $\sup_{\theta\in \Theta}C(\theta)<\infty $ because 
$\Theta$ is a compact set and $\theta \mapsto \|\ddot S_n(\theta)^{-1}\|_{op}$ is continuous.

\subsection{Proof of Theorem~\ref{oracleineq} (nonparametric estimation)}
The proof of the oracle inequality is based on Lemma~\ref{lemadap} and Lemma~\ref{majnun} below. 
The conclusion follows, choosing $2\gamma=\epsilon/(1+\epsilon)$
and $\lambda=\gamma^{-1}\kappa(1-2P)^{-2}=2\kappa(1+\epsilon^{-1})(1-2P)^{-2}$,
{  and $q=(2s_0+1)/3$}.

Let us derive the rate of convergence, which is the second result of Theorem~\ref{oracleineq}. We use the notation of Lemma~\ref{lemadap} and the notation $\|.\|_{\ell}$ for the natural norm of $\ell^2(\C^\Z)$. Let $L\in \mathcal L$. 
 {Since 
$\nu_n(t)= \sum_{l\in \Z} \overline{t_{l}} (\widehat{f^{\star l}}- f^{\star l})$,
$$\sum_{l=-L}^L |\widehat{f^{\star l}}- f^{\star l}|^2=\nu_n({\widehat{f^{\star}_L}-f^{\star}_L})\leq \sup_{t\in B_L}\nu_n(t)\|{\widehat{f^{\star}_L}-f^{\star}_L}\|_{\ell}$$
where we denote $ f_L^\star$ the sequence in $\C^\Z$ such that 
$( f_L^\star)_l={f^{\star l}}$ if $-L\leq l \leq L$ and 0 otherwise.
Hence $\|{\widehat{f^{\star}_L}-f^{\star}_L}\|_{\ell}^2\leq \sup_{t\in B_L}\nu_n(t)\|{\widehat{f^{\star}_L}-f^{\star}_L}\|_{\ell}$ so that 
$\|{\widehat{f^{\star}_L}-f^{\star}_L}\|_{\ell}\leq \sup_{t\in B_L}\nu_n(t)$.
Then, using Lemma~\ref{majnun}
\begin{eqnarray*}
\E\sum_{l=-L}^L |\widehat{f^{\star l}}- f^{\star l}|^2
&=&\E\|{\widehat{f^{\star}_L}-f^{\star}_L}\|_{\ell}^2
\leq  \frac{\kappa}{(1-2P)^2} \frac{2L+1}{n}+\frac{C {(1+R^2)}}{n}\\
&\leq & C' {(1+R^2)}\frac{2L+1}{n}.
\end{eqnarray*}
Using Parseval's identity,
$$\E\|f-\hat f_L\|_2^2=\sum_{|l|>L} |f^{\star l}|^2+C' {(1+R^2)}\frac{2L+1}{n}\leq %
{R^2(1+L^2)^{-s}+C' {(1+R^2)}\frac{2L+1}{n}.}$$
Thus, the oracle inequality gives
\begin{eqnarray*}
\E\|\hat f_{\widehat{L}}-f\|_2^2&\leq & (1+2\epsilon)\min_{L\in \mathcal{L}} \left\{
 {R^2(1+L^2)^{-s}}+(C' {(1+R^2)}+2\lambda)\frac{2L+1}{n}\right\}\\
 &&+\frac{C {(1+R^2)}}{n}\\
 &\leq& C''  {R^2} n^{-2s/(2s+1)}
\end{eqnarray*}
 { 
choosing $L=L_0=\lfloor C n^{1/(2s+1)}\rfloor $. 
This choice is possible since $s\geq s_0$ and then $L_0 $ belongs to $\mathcal{L}$. 
}\\

\begin{lemma}
\label{lemadap}
 Let $\lambda>0$ and $\mathcal{L}$ be a finite set of resolution level and define 
$$\widehat{L}=\underset{L\in \mathcal{L}}{\argmin} \left\{-\sum_{l=-L}^L | \widehat{f^{\star l}}|^2+\lambda \frac{2L+1}{n}\right\}.$$
Then, for all $0<\gamma<1/2$, 
\begin{eqnarray*}
(1-2\gamma)\|\hat f_{\widehat{L}}-f\|_2^2\leq \min_{L\in \mathcal{L}} \left\{(1+2\gamma)\|\hat f_{ L}-f\|_2^2+2\lambda \frac{2L+1}{n}\right\}\\+
 {
\frac{1}{\gamma}\max_{L\in\mathcal{L}}\left( \sup_{t\in B_{L}}\nu_n^2(t)-\lambda \gamma\frac{2L+1}{n}\right)
}
\end{eqnarray*}
where  
 $B_L=\{t\in \C^{\Z}, \, \sum_{l\in \Z} |t_{l}|^2=1, \,t_l=0 \text{ if  }|l|>L \}$ 
and 
$\nu_n(t)= \sum_{l\in \Z} \overline{t_{l}} (\widehat{f^{\star l}}- f^{\star l})$.
\end{lemma}

\begin{proof}
We recall that the dot product $\langle f,g\rangle $ means $\frac{1}{2\pi}\int \overline{f(x)}g(x)dx$ and  that $\|.\|_2$ is the associated norm. Usual Fourier analysis gives for any $L$:
\begin{align*}
\|\hat f_L-f\|_2^2&=-\|\hat f_L\|_2^2+2(\|\hat f_L\|_2^2-\langle \hat f_L,f\rangle )+\|f\|_2^2\\
&=-\sum_{l=-L}^{ L}  | \widehat{f^{\star l}}|^2+2\sum_{l=-L}^L\overline{\widehat{f^{\star l}}}(\widehat{f^{\star l}}-f^{\star l})+\|f\|_2^2\\
&=-\sum_{l=-L}^{L} | \widehat{f^{\star l}}|^2+2\nu_n(\widehat f_L^\star)+\|f\|_2^2
\end{align*}
where we denote $\widehat f_L^\star$ the sequence in $\C^\Z$ such that 
$(\widehat f_L^\star)_l=\widehat{f^{\star l}}$ if $-L\leq l \leq L$ and 0 otherwise.

Now let $L$ be an arbitrary resolution level in $\mathcal{L}$. Using the definition of $\widehat{L}$, 
$$-\sum_{l=-\widehat{L} }^{\widehat{L}} | \widehat{f^{\star l}}|^2+\lambda \frac{2\widehat{L}+1}{n}\leq -\sum_{l=-L}^L  | \widehat{f^{\star l}}|^2+\lambda \frac{2L+1}{n}.$$
Thus 
$$\|\hat f_{\widehat{L}}-f\|_2^2-2\nu_n(\hat f_{\widehat{L}}^\star)+\lambda \frac{2\widehat{L}+1}{n}\leq \|\hat f_L-f\|_2^2-2\nu_n(\hat f_L^\star)+\lambda \frac{2L+1}{n}$$
which leads to 
$$\|\hat f_{\widehat{L}}-f\|_2^2\leq \|\hat f_L-f\|_2^2+2\nu_n(\widehat f_{\widehat{L}}^\star-\widehat f_L^\star)-\lambda \frac{2\widehat{L}+1}{n}+\lambda \frac{2L+1}{n}.$$
But, denoting by $\|.\|_{\ell}$ the natural norm of $\ell^2(\C^\Z)$
\begin{align*}
2\nu_n(\widehat f_{\widehat{L}}^\star-\widehat f_L^\star)&=
2\nu_n\left(\frac{\widehat f_{\widehat{L}}^\star-\widehat f_L^\star}{\|\widehat f_{\widehat{L}}^\star-\widehat f_L^\star\|_{\ell}}\right)\|\widehat f_{\widehat{L}}^\star-\widehat f_L^\star\|_{\ell}\\
2\left|\nu_n(\widehat f_{\widehat{L}}^\star-\widehat f_L^\star)\right|&\leq \gamma\|\widehat f_{\widehat{L}}^\star-\widehat f_L^\star\|_{\ell}^2 + \frac{1}{\gamma}\left|\nu_n\left(\frac{\widehat f_{\widehat{L}}^\star-\widehat f_L^\star}{\|\widehat f_{\widehat{L}}^\star-\widehat f_L^\star\|_{\ell}}\right)\right|^2\\
&\leq 2\gamma (\|\hat f_{\widehat{L}}-f\|_2^2 +
\| f-\hat f_L\|_2^2 )
+ \frac{1}{\gamma}\sup_{t\in B_{L\vee \widehat{L}}}|\nu_n(t)|^2
\end{align*}
where $L\vee \widehat{L}=\max(L, \widehat{L})$. 
Thus 
\begin{align*}
\|\hat f_{\widehat{L}}-f\|_2^2(1-2\gamma) \leq \|\hat f_L-f\|_2^2(1+2\gamma)+ \frac{1}{\gamma}\sup_{t\in B_{L\vee \widehat{L}}}|\nu_n(t)|^2-\lambda \frac{2\widehat{L}+1}{n}+\lambda \frac{2L+1}{n}\\
\leq \|\hat f_L-f\|_2^2(1+2\gamma)+ \frac{1}{\gamma}\left(\sup_{t\in B_{L\vee \widehat{L}}}|\nu_n(t)|^2-\lambda \gamma \frac{2\widehat{L}+2L+2}{n}\right)+2\lambda \frac{2L+1}{n}\\
\leq \|\hat f_L-f\|_2^2(1+2\gamma)+2\lambda \frac{2L+1}{n}+\frac{1}{\gamma}\max_{L'\in\mathcal{L}}\left(\sup_{t\in B_{L'}}|\nu_n(t)|^2-\lambda \gamma \frac{2L'+1}{n}\right).
\end{align*}

\end{proof}

\begin{lemma} \label{majnun}
Assume Assumption  Assumption \ref{hypf} and \ref{hypthetabis}. 
Assume that $f$ belongs to the Sobolev ellipsoid $W(s,R)$
{ 
with $s\geq 1$.
Assume that 
$\mathcal{L}=\{0, \dots, L_n\}$ with  ${L}_n$ such that  $L_n^3\leq C_{\mathcal L}n^{1/q}$ for some $q>1$. 
Then, with the notation of Lemma~\ref{lemadap}, 
for all $\kappa> 3/(2\pi^2)$, 
}
$$\E\max_{L\in \mathcal{L}}\left( \sup_{t\in B_L}|\nu_n(t)|^2-\frac{\kappa}{(1-2P)^2} \frac{2L+1}{n}\right)\leq \frac{C {(1+R^2)}}{n},$$
where $C$ is  a positive constant depending on 
 {$P,q, C_{\mathcal L}, \kappa$}. 
\end{lemma}

\begin{proof}
Denote $R^l=\frac{1}{M^l(\hat \theta)}-\frac{1}{M^l( \theta_0)}$.
First note that 
$$\nu_n(t)=\frac1{2\pi n}\sum_{k=1}^n\sum_{l\in \Z}  \overline{t_{l}} \left(\frac{e^{-ilX_k}}{M^l(\hat \theta)}-
\frac{2\pi g^{\star l}}{M^l( \theta_0)}\right)
=\nu_{n,1}(t)+\nu_{n,2}(t)+\nu_{n,3}(t)$$
where
\begin{align*}
\nu_{n,1}(t)&=\frac1{2\pi n}\sum_{k=1}^n\sum_{l\in \Z}  \overline{t_{l}} \left(\frac{e^{-ilX_k}-2\pi g^{\star l}}{M^l( \theta_0)}\right)\\
\nu_{n,2}(t)&=\frac1{2\pi n}\sum_{k=1}^n\sum_{l\in \Z}  \overline{t_{l} }\left(e^{-ilX_k}-2\pi g^{\star l}\right)R^l\\
\nu_{n,3}(t)&=\frac1{n}\sum_{k=1}^n\sum_{l\in \Z} \overline{ t_{l}} g^{\star l}R^l=\sum_{l\in \Z} \overline{ t_{l}} g^{\star l}R^l.
\end{align*}

{ 
Thus $|\nu_n|^2\leq 3|\nu_{n,1}|^2+3|\nu_{n,2}|^2+3|\nu_{n,3}|^2$, and,
if $\kappa_1=\kappa/3$,
\begin{align*}
&\E\max_L\left( \sup_{ B_L}|\nu_n|^2-\frac{\kappa}{(1-2P)^2} \frac{2L+1}{n}\right)\leq 
3\E\sum_L\left( \sup_{ B_L}|\nu_{n,1}|^2-\frac{\kappa_1}{(1-2P)^2} \frac{2L+1}{n}\right)_+\\
&+
3\E\max_L\left( \sup_{B_L}|\nu_{n,2}|^2\right)
+3\E\max_L\left( \sup_{B_L}|\nu_{n,3}|^2\right)
\end{align*}
where $a_+=\max(a,0)$ denotes the positive part of $a$.\\
}

\textit{\underline{Control of $\nu_{n,3}$}}
First note that 
$$\left|g^{\star l}R^l\right|=\left| f^{\star l}\frac{M^l( \theta_0)-M^l(\hat \theta)}{M^l( \hat\theta)}\right|	\leq \frac{|f^{\star l}|}{1-2P}
\left|M^l( \theta_0)-M^l(\hat \theta)\right|.$$
Thus, using Schwarz inequality
$$\sup_{t\in B_L} |\nu_{n,3}(t)|^2 \leq \sum_{l=-L}^L \frac{|f^{\star l}|^2}{(1-2P)^2}
\left|M^l( \theta_0)-M^l(\hat \theta)\right|^2.$$
Moreover
\begin{eqnarray*}
|M^l( \theta_0)-M^l(\hat \theta)|
&\leq &  \left|(p_0-\hat p)e^{-i\alpha_0 l}+\hat p (e^{-i\alpha_0 l}-e^{-i\hat \alpha l})
+(1-p_0-1+\hat p)e^{-i\beta_0 l}\right.
\\&&+ \left. (1-\hat p) (e^{-i\beta_0 l}-e^{-i\hat \beta l})\right|\\
&\leq  & |p_0-\hat p|+|e^{-i\alpha_0 l}-e^{-i\hat \alpha l}|
+|p_0-\hat p|+|e^{-i\beta_0 l}-e^{-i\hat \beta l}|\\
&\leq & 2 |p_0-\hat p|+|l||\alpha_0-\hat \alpha |+|l||\beta_0 -\hat \beta|\leq 2 |l| \|\theta_0-\hat \theta\|_1
\end{eqnarray*}
(note that it is also true for $l=0$ since $M^0( \theta_0)=M^0(\hat \theta)=1$).

Thus, for any $L\in\mathcal{L}$
$$\sup_{t\in B_L} |\nu_{n,3}(t)|^2\leq \sum_{l=-L_n}^{L_n} \frac{|f^{\star l}|^2}{(1-2P)^2} 4 |l|^2 \|\theta_0-\hat \theta\|_1^2
$$
{ 
Since $s\geq 1$
$$\max_{L\in\mathcal{L}}\left(\sup_{ B_L} |\nu_{n,3}|^2\right) \leq \frac4{(1-2P)^2}\sum_{l=-L_n}^{L_n}{|f^{\star l}|^2}  |l|^{2s} \|\theta_0-\hat \theta\|_1^2
\leq \frac{4R^2}{(1-2P)^2}\|\theta_0-\hat \theta\|_1^2.
$$
According to Proposition~\ref{varthetachap} and inequality $\|x\|_1^2\leq 3 \|x\|^2$,
\label{Ktheta}
there exists a constant $K>0$ such that 
$\E(n\|\hat \theta - \theta_0\|_1^2)\leq K$. Then 
$$\E\max_{L\in\mathcal{L}}\left(\sup_{ B_L} |\nu_{n,3}|^2\right)\leq \frac{C_3R^2}{n}$$
with $C_3=4K/(1-2P)^2$.

\medskip 

\textit{\underline{Control of $\nu_{n,2}$}}
Note that 
$$|R^l |\leq 
\frac{1}{(1-2P)^2}
\left|M^l( \theta_0)-M^l(\hat \theta)\right|
\leq 
\frac{2}{(1-2P)^2}
|l|\|\hat \theta-\theta_0\|_1,$$ so for $t\in B_L$,
\begin{align*}
| \nu_{n,2}(t)|^2 \leq \left( \sum_{l=-L}^{L} |\overline{t_l}(\widehat{g^{\star l}} - g^{\star l})R^l|\right)^2
 \leq \frac{4}{(1-2P)^4}\sum_{l=-L}^{L} |\widehat{g^{\star l}} - g^{\star l}|^2 l^2\|\hat \theta-\theta_0\|_1^2.
\end{align*}
Then, for any $L\in \mathcal{L}$
\begin{align*}
 \sup_{t\in B_L} |\nu_{n,2}(t)|^2
 &\leq \frac{4}{(1-2P)^2}\sum_{l=-L}^{L} |\widehat{g^{\star l}} - g^{\star l}|^2 l^2\|\hat \theta-\theta_0\|_1^2
 \end{align*}
 and
\begin{align*}
\max_{L\in \mathcal{L} }\sup_{t\in B_L} |\nu_{n,2}(t)|^2
 &\leq \frac{4}{(1-2P)^2}\sum_{l=-L_n}^{L_n} |\widehat{g^{\star l}} - g^{\star l}|^2 l^2\|\hat \theta-\theta_0\|_1^2.
 \end{align*}
Using H\"{o}lder's inequality, for any $p,q\geq 1$ such that $\frac1p+\frac1q=1$,
 \begin{align*} 
 \E\max_{L\in \mathcal{L} }\sup_{t\in B_L} |\nu_{n,2}(t)|^2
 & \leq\frac{4}{(1-2P)^2} \sum_{l=-L_n}^{L_n} l^2 \E^{1/p}(|\widehat{g^{\star l}} - g^{\star l}|^{2p})\E^{1/q}\|\hat \theta-\theta_0\|_1^{2q}
 \end{align*}
But Proposition~\ref{varthetachap} gives us 
$$\E\|\hat \theta-\theta_0\|_1^{2q}\leq (1+2\pi+2\pi)^{2q-2}\E\left(3\|\hat \theta-\theta_0\|^{2}\right)\leq 
K'(q)n^{-1}.$$
 Moreover, we can apply the Rosenthal inequality to the variables $Y_k=e^{ilX_k}-\E(e^{ilX_k})$: there exists $C(2p)>0$ such that 
 \begin{eqnarray*}
 \E \left|\sum_{k=1}^n Y_k\right|^{2p}&\leq & C(2p)\left(
 \sum_{k=1}^n \E|Y_k|^{2p}+\left(\sum_{k=1}^n \E|Y_k|^2\right)^{p}
 \right)\\
 &\leq & C(2p)\left(
 n 2^{2p}+\left(4n\right)^{p}
 \right)\leq  C'(p)n^{p}
 \end{eqnarray*}
that  provides 
 $$\E(|\widehat{g^{\star l}} - g^{\star l}|^{2p})=\E\left((2\pi)^{-2p}
\left |\frac1n\sum_{k=1}^n Y_k\right|^{2p}\right)
\leq  (2\pi)^{-2p}C'(p)n^{-p}.$$
 Thus
 \begin{align*} 
 \E\max_{L\in \mathcal{L} }\sup_{t\in B_L} |\nu_{n,2}(t)|^2
 &\leq \frac{4}{(1-2P)^2} \sum_{l=-L_n}^{L_n} l^2 
 (2\pi)^{-2}C'(p)^{1/p}n^{-1}
 K'(q)^{1/q}n^{-1/q}\\
 & \leq \frac{C''(q)}{(1-2P)^2}n^{-1-1/q}L_n^3.
 \end{align*}
 Since $L_n^3\leq C_{\mathcal L} n^{1/q}$, we obtain  
 \begin{align*} 
 \E\max_{L\in \mathcal{L} }\sup_{t\in B_L} |\nu_{n,2}(t)|^2
 & \leq  \frac{C_2}{n}
 \end{align*}
 with $C_2=C''(q)C_{\mathcal L}/(1-2P)^2.$
 }

\medskip

\textit{\underline{Control of $\nu_{n,1}$}}

To control $\nu_{n,1}$ , we need Talagrand's inequality. 

\begin{lemma}\label{InegTalagrand}
Let $X_1,\ldots,X_n$ be i.i.d. random variables, and define $\nu_n(t)=\frac{1}{n}\sum_{k=1}^n\psi_t(X_k)-\mathbb{E}[\psi_t(X_k)]$, for $t$ belonging to a countable class $\mathcal{B}$ of real-valued measurable functions. Then, for $\delta>0$, there exist three constants $c_l$, $l=1,2,3$, such that
\begin{eqnarray}\label{TalagrandIntegre}\displaystyle\mathbb{E}\left[ \left(\sup_{t\in\mathcal{B}}\left|\nu_n\left(t\right)\right|^2-c(\delta)H^2\right)_{+}\right]&\leq& c_1\left\{\frac{v}{n}\exp\left(-c_2\delta\frac{nH^2}{v}\right)\right.\\\nonumber &&\displaystyle\left.+\frac{M_1^2}{C^2(\delta)n^2}\exp\left(-c_3C(\delta)\sqrt{\delta}\frac{nH}{M_1}\right)\right\},\end{eqnarray}
\noindent with $C(\delta)=(\sqrt{1+\delta}-1)\wedge 1$, $c(\delta)=2(1+2\delta)$ and
$$\displaystyle \sup_{t\in\mathcal{B}} \|\psi_t\|_{\infty} \leq M_1\mbox{, }\mathbb{E}\left[ \sup_{t\in\mathcal{B}}\left|\nu_n(\psi_t)\right|\right]\leq H\mbox{, and }\sup_{t\in\mathcal{B}}\mbox{Var}\left(\psi_t\left(X_1\right)\right)\leq v.$$
\end{lemma}
Inequality (\ref{TalagrandIntegre}) is a classical consequence of Talagrand's inequality given in \citet{KR2005}: see for example Lemma 5 (page 812) in \cite{lacour08}. Using density arguments, we can apply it to the unit sphere of a finite dimensional linear space.

Here $\nu_{n,1}(t)=\frac{1}{n}\sum_{k=1}^n\psi_t(X_k)-\mathbb{E}[\psi_t(X_k)]$ with
$$\psi_t(X)=\frac1{2\pi }\sum_{l\in \Z}  \overline{t_{l}} \frac{e^{-ilX}}{M^l( \theta_0)}, \qquad \E(\psi_t(X))=\sum_{l\in \Z}  \overline{t_{l}} \frac{g^{\star l}}{M^l( \theta_0)}$$

Let us compute $M_1, H$ and $v$. 
\begin{itemize}
\item Using Cauchy Schwarz inequality, for $t\in B_L$,
\begin{align*}
|\psi_t(u)|^2&=\left|\frac1{2\pi }\sum_{l=-L}^L  \overline{t_{l}} \frac{e^{-ilu}}{M^l( \theta_0)}\right|^2 \leq \frac1{4\pi ^2}\sum_{l=-L}^L|t_l|^2\sum_{l=-L}^L  \left| \frac{e^{-ilu}}{M^l( \theta_0)}\right| ^2\\
&\leq \frac{1}{4\pi^2(1-2p_0)^2}(2L+1),
\end{align*}
thus $M_1=\frac{1}{2\pi(1-2p_0)}\sqrt{2L+1}.$
\item 
Using Cauchy Schwarz inequality, for $t\in B_L$,
\begin{align*}
\sup_{t\in {B}_L}\left|\frac{1}{2\pi n}\sum_{k=1}^n\sum_{l\in \Z} \overline{ t_{l}} \left(\frac{e^{-ilX_k}}{M^l(\theta_0)}-
\E\left(\frac{e^{-ilX_k}}{M^l( \theta_0)}\right)\right)\right|^2\\
\leq \sum_{l=-L}^L\left|\frac{1}{2\pi n}\sum_{k=1}^n\left(\frac{e^{-ilX_k}}{M^l(\theta_0)}-
\E\left(\frac{e^{-ilX_k}}{M^l( \theta_0)}\right)\right)\right|^2,
\end{align*}

then 
\begin{align*}
\E\left(\sup_{t\in {B}_L}\left|\nu_{n,1}(\psi_t)\right|^2\right)
 & \leq \sum_{l=-L}^L \Var \left(
\frac{1}{2\pi n}\sum_{k=1}^n\frac{e^{-ilX_k}}{M^l(\theta_0)}\right)
\leq \sum_{l=-L}^L \frac{1}{4\pi^2 n} \Var \left(
\frac{e^{-ilX_1}}{M^l(\theta_0)}\right)\\
& \leq  \frac{1}{4\pi^2 n} \sum_{l=-L}^L\E\left|
\frac{e^{-ilX_1}}{M^l(\theta_0)}\right|^2 
\leq  \frac{1}{4\pi^2(1-2p_0)^2} \frac{2L+1}{ n},
\end{align*}
thus by Jensen’s inequality $H^2=\frac{1}{4\pi^2(1-2p_0)^2} \frac{2L+1}{ n}$.
\item It remains to control the variance. If $t\in B_L$
\begin{align*}
\Var(\psi_t(X))\leq \E\left| \frac1{2\pi }\sum_{l=-L}^L \overline{ t_{l}} \frac{e^{-ilX}}{M^l( \theta_0)}\right|^2
=\frac1{4\pi^2 }\sum_{l,l'}  t_{l}\overline{t_{l'} }\frac{\E(e^{-ilX}\overline{e^{-il'X}})}{M^l( \theta_0)\overline{M^{l'}( \theta_0)}} \\
=\frac1{2\pi }\sum_{l,l'}  t_{l}\overline{t_{l'} }\frac{g^{\star(l-l')}}{M^l( \theta_0)\overline{M^{l'}( \theta_0)}} 
\end{align*}
Using twice Schwarz inequality 
\begin{align*}
\Var(\psi_t(X))&\leq \frac1{2\pi }\sqrt{\sum_{l}\left|\frac{t_{l}}{M^l( \theta_0)}\right|^2\sum_l\left|\sum_{l'}  \frac{\overline{t_{l'} }}{\overline{M^{l'}( \theta_0)}}g^{\star(l-l')}\right|^2} \\
&\leq \frac1{2\pi (1-2p_0)}\sqrt{\sum_l\left|\sum_{l'}  \frac{\overline{t_{l'} }}{\overline{M^{l'}( \theta_0)}}g^{\star(l-l')}\right|^2} \\
&\leq \frac1{2\pi (1-2p_0)}\sqrt{\sum_l\sum_{l'}  \left|\frac{\overline{t_{l'} }}{\overline{M^{l'}( \theta_0)}}\right|^2\sum_{l'}\left|g^{\star(l-l')}\right|^2} \\
&\leq \frac1{2\pi (1-2p_0)}\sqrt{\sum_l\frac{1}{|1-2p_0|^2}\sum_{j\in \Z}\left|g^{\star j}\right|^2} \\
&\leq \frac{ { R}}{2\pi (1-2p_0)^2}\sqrt{2L+1}, \\
\end{align*}
since $\sum_{j\in \Z}\left|g^{\star j}\right|^2\leq \sum_{j\in \Z}\left|f^{\star j}\right|^2 {\leq R^2}$.
Thus $v=\frac{ { R}}{2\pi (1-2p_0)^2}\sqrt{2L+1} $.
\end{itemize}

 Inequality (\ref{TalagrandIntegre})  becomes
 
 \begin{eqnarray*}
& &\displaystyle\E\left[ \left(\sup_{t\in {B}_L}\left|\nu_{n,1}\left(t\right)\right|^2-\frac{c(\delta)}{4\pi^2(1-2p_0)^2} \frac{2L+1}{ n}\right)_{+}\right]\\
&& \leq
 c_1\left\{\frac{R\sqrt{2L+1} }{2\pi (1-2p_0)^2n}\exp\left(-c_2\delta\frac{\sqrt{2L+1}}{2\pi R}\right)\right.\\
&&\qquad \displaystyle\left.+\frac{2L+1}{4\pi^2(1-2p_0)^2C^2(\delta)n^2}\exp\left(-c_3C(\delta)\sqrt{\delta n}\right)\right\}\\
&& \leq \frac{K\max(R,1)}{n}
 \left\{\sqrt{2L+1} \exp\left(-c\sqrt{2L+1}\right)+\frac{2L+1}{n}\exp\left(-c\sqrt{ n}\right)\right\}
 \end{eqnarray*}
with $K$ 
and $c$ positive constants depending on $P,c_1,c_2,c_3, \delta$.
 This ends the control of $\nu_{n,1}$ with $\kappa_1=\frac{c(\delta)}{4\pi^2}$ since 
 $$\sum_{L\in \mathcal{L}}\left\{\sqrt{2L+1} e^{-c\sqrt{2L+1}}+\frac{2L+1}{n}e^{-c\sqrt{ n}}\right\}\leq \sum_{L=0}^{\infty}\sqrt{2L+1} e^{-c\sqrt{2L+1}}+\sharp\mathcal{L} e^{-c\sqrt{ n}}=
 O(1).$$
{ 
Finally it is sufficient to take
 $$\kappa\geq 3\kappa_1=
 \frac{3(2+4\delta)}{4\pi^2}=
\frac{3}{2\pi^2}+\frac{3\delta}{\pi^2}$$
 to conclude the proof. Since  $\delta$ can be chosen arbitrary small, and we have assumed  $\kappa>3/(2\pi^2)$, this condition is satisfied.
 }

\end{proof}

\section*{Acknowledgement}

The authors would like to thank the Editors and one anonymous referee for valuable comments and suggestions leading to corrections and improvements of the article.


\bibliographystyle{apalike} 
\bibliography{biblio}       


\end{document}